\documentclass[a4paper, 11pt]{amsart}

\usepackage{enumerate}
\usepackage{mathrsfs}
\usepackage{float}
\usepackage{amssymb,amsmath}
\usepackage[all]{xy}
\usepackage{xcolor}
\usepackage{tikz}

\title[]{On the Mori theory and Newton-Okounkov bodies of Bott-Samelson varieties}
\author{Georg Merz, David Schmitz, Henrik Sepp\"anen} 

\thanks{The first author was supported by DFG Research Training Group 1493 ``Mathematical Structures in Modern Quantum Physics''.}

\keywords{Effective cone, Bott-Samelson variety, Mori dream space, Newton-Okounkov body}

\address{Georg Merz,
Mathematisches Institut,
Georg-August-Univer\-sit\"at G\"ot\-tingen,
Bunsenstra\ss e 3-5, 
D-37073 G\"ottingen,
Germany}
\email{georg.merz@mathematik.uni-goettingen.de}

\address{David Schmitz,
Mathematisches Institut,
Universit\"at Bayreuth,
D-95447 Bayreuth,
Germany}
\email{david.schmitz@uni-bayreuth.de}

\address{Henrik Sepp\"{a}nen,
Mathematisches Institut,
Georg-August-Univer\-sit\"at G\"ot\-tingen,
Bunsenstra\ss e 3-5, 
D-37073 G\"ottingen,
Germany}
\email{henrik.seppaenen@mathematik.uni-goettingen.de}

\newcommand{\C}{\mathbb{C}}
\newcommand{\R}{\mathbb{R}}
\newcommand{\N}{\mathbb{N}}
\newcommand{\Z}{\mathbb{Z}}
\newcommand{\Q}{\mathbb{Q}}
\renewcommand{\P}{\mathbb{P}}

\renewcommand{\phi}{\varphi}

\renewcommand{\tilde}{\widetilde}

\newcommand{\Nef}{\mbox{Nef}}
\newcommand{\Pic}{\mbox{Pic}}
\newcommand{\Eff}{\overline{\mbox{Eff}}}

\newcommand{\vol}{\mbox{vol}}
\newcommand{\bs}{\mathbb B}

\newcommand{\proj}{\mbox{Proj}}

\renewcommand{\to}{\longrightarrow}

\newtheorem{prop}{Proposition}[section]
\newtheorem{lemma}[prop]{Lemma}
\newtheorem{cor}[prop]{Corollary}
\newtheorem{thm}[prop]{Theorem}
\newtheorem{thmx}{Theorem}

\theoremstyle{definition}
\newtheorem{defin}[prop]{Definition}
\newtheorem{rem}[prop]{Remark}

\newtheorem{remark}[prop]{Remark}
\newtheorem{question}[prop]{Question}
\newtheorem{conjecture}[prop]{Conjecture}

\definecolor{plum}{RGB}{94,31,157}
\definecolor{brown}{RGB}{144,48,0}
\definecolor{magenta}{RGB}{131,7,144}
\definecolor{green}{RGB}{0,91,0}

\newcommand{\tikzcircle}[2][red,fill=red]{\tikz[baseline=-0.5ex]\draw[#1,radius=#2] (0,0) circle ;}%

\begin{document}
\begin{abstract}
 We prove that on a Bott-Samelson variety $X$ every movable divisor is nef. This enables us to consider Zariski decompositions of effective divisors, which in turn yields a 
 description of the Mori chamber decomposition of the effective cone. This amounts to information on all possible birational morphisms from $X$. Applying this result, we 
 prove the rational polyhedrality of the global Newton-Okounkov body of a Bott-Samelson variety with respect to the so called `horizontal' flag. In fact, we prove the 
 stronger property of the finite generation of the corresponding global value semigroup. 
 \end{abstract}
\maketitle

%

\section{Introduction}
Bott-Samelson varieties arise naturally from the study of flag varieties as resolutions of singularities of Schubert varieties (see e.g. \cite{D74}). Their line bundles have been studied by Lauritzen and Thomsen (\cite{LT04}), and by Anderson (\cite{and}). If $X_w$ is a Bott Samelson variety corresponding to a reduced word $w$, an explicit finite set of generators for the cone of effective/nef divisors is described by the first two authors.

In addition to the rational polyhedrality of the effective/nef cone, Anderson (\cite{and}) and 
the last two authors of this article (\cite{SS17}) showed independently that a Bott-Samelson variety $X=X_w$ is log-Fano. In particular, $X$ is a Mori dream space in the 
sense of \cite{hk}. Consequently, the Cox ring Cox$(X)=\bigoplus_{D\in \mbox{\tiny Div}(X)}H^0(X,\mathcal O_X(D))$ is 
a finitely generated $\C$-algebra.  Moreover, there are only finitely many contracting birational maps from $X$. These 
contractions correspond to a decomposition of the effective cone into finitely many subcones, the so called Mori chambers. More concretely, in \cite{hk} two big divisors $D_1$ and $D_2$ are 
called \emph{Mori-equivalent}, if their induced maps $X\dasharrow \proj(\bigoplus_mH^0(X,\mathcal{O}_X(mD_i)))$ agree. This can be rephrased by saying that their respective Minimal Model Programs (MMP) 
coincide. The closures in the N\'eron-Severi vector space $N^1(X)_\R$ of the equivalence classes are the aforementioned Mori chambers. Despite their straightforward definition, these chambers are 
very hard to determine in almost all concrete cases. 
Especially the existence of small contractions still poses a plethora of challenges one of which the two 
last authors faced when studying  global Newton-Okounkov bodies of Mori dream spaces in \cite{SS17}.
 In their work the main complication consisted in the fact that restricted volumes of divisors with small base loci, i.e. of codimension higher than one, behave unpredictably compared to 
 those of nef divisors. Another consequence of the appearance of small base loci is the fact that in order to define Zariski decompositions of effective divisors, one has to consider higher birational models 
 on which these base loci are resolved. In practice, it is usually very difficult to control these resolutions.

In this article, we apply results about Newton-Okounkov bodies from \cite{KL15} and \cite{SS17} to show that in the case of Bott-Samelson varieties such problems need not concern us. Concretely, we prove the following. 
\begin{thmx}
\label{thma}
 Let $X=X_w$ be a Bott-Samelson variety corresponding to a reduced sequence $w$. Then every movable divisor on $X$ is base-point-free, and hence 
 \begin{align*}
  \mbox{Mov}(X)=\mbox{Nef}(X).
 \end{align*}
\end{thmx}
At a first glance it may seem that, by realizing this fact, the struggle in \cite{SS17} to deal with small modifications on Bott-Samelson varieties 
was in vain. Note however that in order to prove the above theorem we rely on a result from  \cite{SS17} for which the existence of small contractions could not be ruled out from the outset.  

Once Theorem \ref{thma} is established, we have access to Zariski decompositions, which then just consist of the decomposition of a big divisor into its \mbox{($\Q$--)movable} and ($\Q$--)fixed parts. 
This enables us to describe the pseudo-effective cone of a Bott-Samelson variety in detail and in particular to give criteria for divisor classes to span a common Mori chamber. It will turn 
out that Mori chambers are uniquely determined by stable base loci occurring in their interior. More concretely, we prove the following.
\begin{thmx}
Let $X=X_w$ be a Bott-Samelson variety for a reduced word $w$. Then each Zariski chamber defines a Mori chamber and vice versa.

\end{thmx}

As an application of Theorem \ref{thma}, we turn to Newton-Okounkov bodies on Bott-Samelson varieties and the question of finite generation of the semigroups of valuation vectors coming up in the construction. 

The question of finite generation of this semigroup has been intensely studied in the last years since D. Anderson's observation in \cite{A13} that the finite 
generation of the value semigroup of an ample divisor implies the existence of a toric degeneration and the  construction of a related integrable system in this situation in \cite{HaKa}. 

Also the study of Newton-Okounkov bodies on Bott-Samelson varieties has recently become an active field of research.
In \cite{A13} a particular Bott-Samelson variety is considered as an example. A more thorough analysis of Newton-Okounkov bodies for Bott-Samelson varieties was initiated by Kaveh in \cite{K11}, where he showed that 
Littelmann's string polytopes (cf. \cite{Li98}) can be realized as Newton-Okounkov bodies for divisors with respect to a certain valuation. This valuation is however not defined by a flag of subvarieties in terms of order of vanishing. 
In \cite{HY}, the authors describe Newton-Okounkov bodies of Bott-Samelson varieties for divisors $D$ satisfying a certain condition.
In contrast to Kaveh's work, they use a flag to define the valuation, which we will call the `horizontal' flag. In particular they prove the finite generation of the value semigroup in this context. In \cite{SS17}, the rational polyhedrality of the global Newton-Okounkov with respect to the so-called `vertical' flag was proven.

Although both the `vertical'- and the `horizontal' flag consists of Bott-Samelson varieties, their embeddings into $X$ are very different: whereas the divisor 
$Y_1$ in the `vertical' flag is a fibre of a bundle $X \to \P^1$, and thus moves in a natural family, the divisor in the horizontal flag is  fixed, i.e., it is the only element in its linear system.

In this article, we consider the `horizontal' flag and generalize the results of \cite{HY} to all effective divisors $D$. We prove the rational polyhedrality of the global Newton-Okounkov 
body, similarly as in \cite{SS17}. Note, however, that our proof will be substantially less technical since we can make use of Theorem A and the fact, 
derived in \cite{LT04} (see Lemma \ref{lemcohomology}), that the restriction morphism of global sections of nef divisors to  $Y_1$ is surjective. This property is significantly stronger than 
the corresponding identity of restricted volumes $\text{vol}_{X|Y_1}(D)=\text{vol}_{Y_1}(D|Y_1)$ which holds for the `vertical' flag.
Indeed, it will give us the following result.
\begin{thmx}
\label{thmc}
Let $X=X_w$ be a Bott-Samelson variety for a reduced word $w$, and let $Y_\bullet$ be the horizontal flag. Then, the semigroup
\begin{align*}
\Gamma_{Y_\bullet}(X_w):=\{(\nu(s),D) \ \vert \ D\in \Pic(X_w), \ s\in H^0(X,\mathcal{O}_X(D))\setminus\{0\} \}
\end{align*}
is finitely generated.
\end{thmx} 
To our best knowledge, apart from the toric case, no known examples of varieties admitting a finitely generated global semigroup $\Gamma_{Y_\bullet}(X)$ have been studied 
in the literature so far.  It also remains unclear to us whether the above theorem also holds with the `vertical' flag. 

Note also that our result goes in line with the recent work of Postinghel 
and Urbinati (\cite{PU16}). Apart from also 
showing that for each Mori dream space $X$ there is a flag on a birational model of $X$ such that the  corresponding global Newton-Okounkov body is rational polyhedral, they prove the finite 
generation of the value semigroup $\Gamma(D)$ of any big divisor $D$.
However, their work does not imply that the global semigroup $\Gamma(X)$ is finitely generated.

It is a well-known fact that the finite generation of the value semigroup, does not induce a toric degeneration to a  \emph{normal} toric variety.
This is directly related to the fact that the value semigroup itself need not be normal despite being finitely generated. Hence, it would be desirable to find a 
criterion for normality of value semigroups. We prove a sufficient criterion in the case of Bott-Samelson varieties. Namely, if the Zariski decomposition of any integral effective divisor
on $X$ is integral, then the normality of the global value semigroup $\Gamma_{Y_\bullet}(X)$ with respect to the `horizontal' flag follows. Thus, in this case any ample divisor yields a 
degeneration to a normal toric variety.

We can generalize the picture described so far to the setting of flag varieties and Schubert varieties contained therein. Given a parabolic 
subgroup $P\subset G$ containing $B$, and a reduced expression $w=(s_1,\dots,s_n)$ such that we have a birational resolution
$
	p: X_w\to Z_{\overline w}
$
of the Schubert variety corresponding to $\overline w$, the `horizontal' flag on $X_w$ induces a $\Z^n$-valued valuation-like function $\nu$ on $Z_{\overline w}$. Then, 
the following is a consequence of Theorem \ref{thmc}.
\begin{thmx}
	Let $Z_{\overline{w}} \subseteq G/P$ be the Schubert variety for the reduced word $w$, then the global semigroup $\Gamma_\nu(Z_{\overline{w}})$ is finitely generated. In 
	particular, $\Delta_\nu(Z_{\overline{w}})$ is rational polyhedral.
\end{thmx}

Consequently, for any partial flag variety $G/P$ the global semigroup $\Gamma_{\nu}(G/P)$ is finitely generated.

In order to illustrate the results of this paper, we apply them to two concrete Bott-Samelson varieties given as incidence varieties. Their Mori chamber structure is described 
in Sections \ref{s:ex1}, whereas the generators of their global value semigroups and their global Newton-Okounkov bodies are determined in Section \ref{s:ex2}. 

\section*{Acknowledgements} We thank Marcel Maslovari\'c for valuable discussions.

\section{Preliminaries and notation}\label{secprelim}

\subsection{Bott-Samelson varieties}
Let $G$ be a connected and simply connected reductive complex linear group,
let $B \subseteq G$ be a Borel subgroup, and let $W$ be the Weyl group of $G$.
Then for a sequence $w=(s_1,\dots,s_n)$ in $W$, we can associate the Bott-Samelson variety $X_w$ as follows.
Let $P_i$ be the  minimal parabolic subgroup containing $B$ corresponding to the simple reflection $s_i$.
Let $P_w:=P_1 \times \cdots \times P_n$ be the product of the corresponding parabolic subgroups, 
and consider the right action of $B^n$ on $P_w$ given by
\begin{align*}
(p_1,\ldots, p_n)(b_1,\ldots, b_n):=(p_1b_1, b_1^{-1}p_2b_2,b_2^{-1}p_3b_3,\ldots, b_{n-1}^{-1}p_nb_n).
\end{align*}
The Bott-Samelson variety $X_w$ is the quotient
\begin{align*}
X_w:=P_w/B^n=P_1 \times^B (P_2 \times^B \times \cdots \times^B P_n).
\end{align*}
We can represent points in $X_w$ by  tuples $[(p_1,\dots,p_n)]$ for $p_i\in P_i$ and the square brackets denote taking the class in the quotient.
For more details on this construction we refer to \cite{LT04} or \cite{SS17}. 

Originally, Damazure constructed these varieties for a sequence $w$ which is reduced \cite{D74}. In this case, he proved that $X_w$ is a desingularization of the Schubert variety $Z_{\overline{w}}$. In this article, whenever we talk about a Bott-Samelson variety $X_w$, the sequence $w$ will be assumed to be reduced.

Let now $X=X_w$ be a Bott-Samelson variety of dimension $n$. Then $\Pic(X)\cong \mathbb{Z}^n$. 
There are two important bases of $\Pic(X)$.
The first one is called the \emph{effective basis}. It consists of prime divisors $E_1,\dots, E_n$, which can be defined inductively. Justifying its name, the cone spanned by this basis in $N^{1}(X)_{\R}$ is the  cone of effective divisor classes.
The second basis, which we call the \emph{$\mathcal{O}(1)$-basis}, will be denoted by $D_1,\dots, D_n$. These divisors also generate $\Pic(X)$ as a group, whereas the cone they span coincides with the nef cone $\Nef(X)$. 
Note that since $\Pic(X)=N^1(X)$, we will not explicitly distinguish between the divisor $D$ and its class $[D]$.

\subsection{Newton-Okounkov bodies}
For the theory of Newton-Okounkov bodies  we follow the notation and conventions of \cite{LM09}. 
In particular, a flag of irreducible subvarieties $Y_{\bullet} : X = Y_{0} \supseteq Y_{1} \supseteq \dots \supseteq Y_{n-1} \supseteq Y_{n} = \{pt\}$ is called \emph{admissible} if 
$Y_n$ is a smooth point on each $Y_i$.
Each admissible flag $Y_\bullet$ gives rise to a valuation-like function 
\begin{align*}
\nu_{Y_\bullet}\colon \bigsqcup_{D\in \Pic(X)} H^0(X,\mathcal{O}_X(D))\setminus \{0\} \to \mathbb{Z}^d.
\end{align*}
Then for a big divisor $D$, we can define the value semigroup
\begin{align*} 
 \Gamma_{Y_\bullet}(D)=\{(\nu_{Y_\bullet}(s),k) \ \vert \ k\in \mathbb{N}, \ s\in H^0(X,\mathcal{O}(kD))\setminus\{0\}\}.
\end{align*}
The
 \emph{Newton-Okounkov body} of $D$ with respect to the flag $Y_\bullet$ as
\begin{align*}
\Delta_{Y_\bullet}(D)=\overline{\text{Cone}(\Gamma_{Y_\bullet}(D))}\cap (\mathbb{R}^d\times\{1\}).
\end{align*}
It is proven in \cite{LM09} that $\Delta_{Y_\bullet}(D)$ only depends on the numerical class of $D$ in $N^1(X)$.

Similarly, if $Y$ is a closed subvariety, then  $\Delta_{X\vert Y}(D)$ denotes the Newton-Okounkov body of the graded linear system 
$W_\bullet$ with $$W_k=\mbox{Im}(H^0(X,\mathcal{O}_X(kD))\to H^0(Y,\mathcal{O}_Y(kD)))$$ with 
respect to a fixed flag on $Y$.

Furthermore, there exists a closed convex cone 
\begin{align*}
	\Delta_{Y_\bullet}(X)\subset  \R^n\times N^1(X)_\R
\end{align*}
such that for each big divisor $D$ the fibre of the second projection over $[D]$ 
is exactly $\Delta_{Y_\bullet}(D)$. We call $\Delta_{Y_\bullet}(X)$ the \emph{global Newton-Okounkov body}.

In case we have $\Pic(X)=N^1(X)$, e.g. if $X$ is a Bott-Samelson variety,  we can define the global Newton-Okounkov body using the \emph{global semigroup}
\begin{align*}
\Gamma_{Y_\bullet}(X):=\{(\nu(s),D) \ \vert \ D\in \Pic(X)=N^1(X), \ s\in H^0(X,\mathcal{O}_X(D))\setminus\{0\} \}
\end{align*}
Then, the global Newton-Okounkov is given by $\Delta_{Y_\bullet}(X)=\overline{\text{Cone}(\Gamma_{Y_\bullet}(X))}$.

Now we consider Newton-Okounkov bodies of effective but not necessarily big divisors.
There are two different ways to define Newton-Okounkov bodies in this situation. One way is to just define them via the valuation-like function $\nu_{Y_\bullet}$. More 
concretely, for an effective $\mathbb{Q}$-divisor $D$ on $X$, we define the \emph{valuative Newton-Okounkov body} as
\begin{align*}
\Delta^{val}_{Y_\bullet}(D):=\frac{1}{k}\cdot\overline{\text{Cone}(\Gamma_{Y_\bullet}(kD))}\cap \left(\mathbb{R}^n\times \{1\} \right)
\end{align*}
where $k\in \mathbb{Z}$ is chosen such that $kD$ is integral. Note that $\Delta_{Y_\bullet}^{val}(D)$ is in general not well-defined for numerical classes and does indeed depend on the linear equivalence class. 
On the other hand, we can also define a Newton-Okounkov body by considering the global Newton-Okounkov body and then taking a fibre over a divisor.
So, we define the \emph{numerical Newton-Okounkov body} as
\begin{align*}
\Delta^{num}_{Y_\bullet}(D):=\Delta_{Y_\bullet}(X)\cap \left( \mathbb{R}^n\times\{D\} \right).
\end{align*}

Note that, in general $\Delta_{Y_\bullet}^{num}(D)\neq \Delta_{Y_\bullet}^{val}(D)$. However, if $D$ is big, both definitions coincide and we just write $\Delta_{Y_\bullet}(D)$.

\section{The movable cone}
Before we are able to prove Theorem \ref{thma}, we need the following.
\begin{lemma}\label{lemma:baseloci} 
Let $D$ be a big divisor on a Mori dream space $X$. Let furthermore $\iota: Y\hookrightarrow  X$ be a closed subvariety of $X$ which is itself a Mori dream space and assume 
that $\iota^*D$ is big and nef. Let us denote by $W_\bullet$ the restricted graded linear series of $D$ to $Y$.
If the identity of volumes $\vol_{X\vert Y}(D)=\vol_{Y}(\iota^*D)$ holds, then the stable base loci also agree, i.e.
\begin{align*}
\bs(W_\bullet)=\bs(\iota^*D)=\emptyset.
\end{align*}
\end{lemma}

\begin{proof}
Since $Y$ is a Mori dream space, $\iota^*D$ is semiample, i.e. $\bs(\iota^*D)=\emptyset$.
Let us assume that $\bs(W_\bullet)$ is not empty and choose $P\in \bs(W_\bullet)$. Let $Y_\bullet$ be an admissible flag on $Y$ which is centered at the point $P$.
By \cite[Theorem A]{KL15}, it follows that the origin is contained in the Newton-Okounkov body $\Delta_{Y_\bullet}(\iota^*D)$. 
Clearly, $\Delta_{Y_\bullet}(W_\bullet)$ is contained in $\Delta_{Y_\bullet}(\iota^*D)$.  Since we have an equality of volumes, $\text{vol}_{X\vert Y}(D)=\vol_{Y}(\iota^*D)$, this inclusion is in fact an equality, i.e., 
\begin{align*}
\Delta_{Y_\bullet}(\iota^*D)=\Delta_{Y_\bullet}(W_\bullet)=\Delta_{X\vert Y}(D).
\end{align*}
As $W_\bullet$ is the restricted graded linear series of the finitely generated divisor $D$, it is finitely generated. We can assume without loss of generality that it is generated in degree $1$. 
Seeing that $0$ lies in $\Delta_{X\vert Y}(D)$, it follows from the construction that there is sequence $(m_k)_{k \in \N}$ of natural numbers, with $\lim_{k \to \infty} m_k=\infty $, and a sequence of sections $(s_k)_{k\in\mathbb{N}}$ such that $s_k\in W_{m_k}$ and 
\begin{align*}
1/m_k \cdot \nu(s_k)=1/m_k \cdot (\nu_1(s_k),\dots, \nu_n(s_k))\to 0, \quad \text{as } k\to \infty.
\end{align*}
This implies that $1/m_k \sum_{i=1}^n \nu_i(s_k) \to 0$.
However, by \cite[Lemma 2.4]{KL15}, 
\begin{align*}
1/m_k\cdot \text{ord}_P(s_k)\leq 1/m_k\cdot \sum_{i=1}^n \nu_i(s_k).
\end{align*}
Since $W_\bullet$ is generated in degree one, all sections $s\in W_{m_k}$ vanish at the point $P$ to order at least $m_k$.
As a consequence, the left hand side is bounded from below by $1$. This, however, contradicts the fact that the right hand side tends to $0$. Hence, $P$ cannot lie in $\bs(W_\bullet)$.
\end{proof}

\begin{thm}
 Let $X=X_w$ be a Bott-Samelson variety. Then every movable divisor on $X$ is base-point-free, and hence 
 \begin{align*}
  \mbox{Mov}(X)=\mbox{Nef}(X).
 \end{align*}
\end{thm}

\begin{proof}
 We prove the claim by induction on $n=\mbox{dim}\, X$. If $n=1$, then $X=\mathbb{P}^1$ and the claim obviously holds. 
 
 Assume now that it holds for $n-1$, and write $X$ as the fibre bundle $X=P \times^B Y$ with projection 
 $$\pi: X \to P/B=\mathbb{P}^1, \quad \pi([p, y]):=pB,$$ 
 where $Y=\pi^{-1}(pB)$ is a Bott-Samelson variety of dimension $n-1$. The group $P$ acts on $X$ by 
 \begin{align*}
  (p, [p', y]):=[pp', y], 
 \end{align*}
 and the projection $\pi$ is clearly $P$-equivariant. 
 
 Let $D$ be a movable divisor on $X$. Since $P$ acts naturally on the line bundle $\mathcal{O}_X(D)$, and hence on the section space $H^0(X, \mathcal{O}_X(D))$ and the 
 section ring $R(X, \mathcal{O}_X(D))$ (cf. \cite{SS17}), the stable base locus $\mathbb{B}(D)$ is $P$-invariant. We therefore have 
 \begin{align} 
  \mathbb{B}(D)=P(\mathbb{B}(D) \cap Y), \label{E: invbaseloc} 
 \end{align}
and 
\begin{align*}
 \mbox{codim}(\mathbb{B}(D), X)=\mbox{codim}(\mathbb{B}(D) \cap Y, Y). 
\end{align*}
In particular, the restriction $D\-{\mid_Y}$ is a movable divisor on $Y$, and hence $D\-{\mid_Y}$ is base point-free by induction, i.e., 
\begin{align*}
 \mathbb{B}(D\-{\mid_Y})=\emptyset.
\end{align*}

We now assume that $D$ is both movable and big. By \cite[Proposition 3.1]{SS17} we then have the identity of volumes
\begin{align*}
 \mbox{vol}_{Y} (D\-{\mid_Y})=\mbox{vol}_{X\vert Y}(D), 
\end{align*}
where the right hand side denotes the volume of the restricted linear series. From this, we deduce by Lemma \ref{lemma:baseloci} that 
\begin{align*}
 \mathbb{B}(D) \cap Y=\mathbb{B}(D\-{\mid_Y})=\emptyset.
\end{align*}
The identity \eqref{E: invbaseloc} now shows that $D$ is base-point-free, wich implies that $D$ is nef. 

Finally, if $D$ is merely movable, an approximation by big and movable divisors yields that $D$ is nef. 
\end{proof}

\section{Mori chamber decomposition of Bott-Samelson varieties}
In this section we give an explicit description of the Mori chamber decomposition of a Bott-Samelson variety.

\subsection{Zariski decomposition}
We have seen in the previous section that for a Bott-Samelson variety $X_w$ the nef cone and the movable cone coincide. One consequence of 
this fact is that Nakayama's $\sigma$-decomposition of pseudo-effective divisors $D$ as $D=P_\sigma(D) + N_\sigma(D)$ into movable and fixed part (c.f. \cite{N04})  indeed gives a Zariski decomposition, i.e., the 
positive part $P_\sigma(D)$ is automatically nef.
In this situation, we say that $X$ \emph{admits Zariski decompositions} and we write $D=P(D) +N(D)$ for the positive and negative part, respectively. 

\begin{rem}\label{remmax}
The negative part of the  $\sigma$-decomposition is characterized by being 
the minimal subdivisor $N(D)$ of $D$ such that $D-N(D)$ is movable \cite[Proposition 1.14 (2)]{N04}.
If $X$ admits Zariski decompositions, this means that the negative part is the minimal 
subdivisor $N(D)$ of $D$ such that $D-N(D)$ is nef. Or differently stated, the 
positive part $P(D)$ of $D$ is the maximal subdivisor of $D$ such that $P(D)$ is nef.
\end{rem}

\subsection{Zariski chambers}
In this section we define and describe Zariski chambers on a Bott-Samelson variety. The definition of Zariski chambers is analogous to the surface case introduced in \cite{BKS04}. 
Note, however, that in order to prevent complications on the boundary, we pass to the closure of equivalence classes. This choice is not essential in the remainder, as in 
order to identify Zariski chambers with Mori chambers we will have to consider closures in any case.

\begin{defin} Suppose $X$ admits Zariski decompositions.
We say that two effective divisors $D$ and $D'$  on $X$ are \emph{Zariski equivalent}, if
\begin{align*}
  \text{supp}(N(D))=\text{supp}(N(D')).
\end{align*}
We denote the closure of the equivalence class of $D$ by $\Sigma_D$.
In case that $\Sigma_D$ contains an open set, we call it a  \emph{Zariski chamber}.  
\end{defin}

\begin{rem}
	Note that most of the pleasant characteristics of Zariski chambers discovered in \cite{BKS04} in the surface case carry over to 
	general $X$ admitting Zariski decompositions. In particular, on 
	the interior of each chamber the augmented base locus $\mathbb B_+(D)$ of a divisor $D$ equals the support of the negative part $N(D)$ and thus these loci are constant on the interior of a chamber. Furthermore, just as in the surface case, the volume of a big divisor is given by the top 
	self-intersection of its positive part and therefore the volume varies polynomially on the interior of each Zariski chamber. 
	
	On the other hand, we do not claim that in general Zariski chambers should be locally polyhedral on the big cone, or even convex. 
	In fact there are examples of varieties that admit Zariski decompositions but have a Zariski chamber which is not convex. As an 
	example consider the blowup of $\P^3$ in two intersecting lines $\ell_1,\ell_2$ and blow up further along the strict transform of the 
	line in $E_1$ corresponding to $\ell_2$. Then every movable divisor is nef and the exceptional divisor $E_3$ is the exceptional 
	locus of two different contractions (corresponding to the different rulings). Its Zariski chamber is the non-convex cone over the 
	classes $E_3, H, H-E_1, H-E_2, 2H-E_1-E_2+E_3$.
\end{rem}

In order to describe the Zariski chambers on Bott-Samelson varieties, we need the following sequence of lemmata. 
For these, we first recall that  a divisor $D$ is called \emph{fixed} if all sufficiently divisible multiples of $D$ are effective and constitute their complete linear series, i.e., 
$H^0(X,\mathcal{O}_X(mD))=\C\cdot s_{mD}$, or $\vert mD\vert =mD$.

\begin{lemma}
Let $X$ be a Mori dream space. For each face $F$ of the effective cone $\text{Eff}(X)$ either all divisors of $F$ are fixed, or all divisors of $\text{relint}(F)$ are not fixed.
\end{lemma}
\begin{proof}
Let $D\in \text{relint}(F)$ be a divisor which is not fixed. Let $M$ be an arbitrary divisor in $\text{relint}(F)$. Then there are positive integers $k$ and $\ell$ such that $L:=k\cdot M-\ell D$ still 
lies in $F$. We can write $k\cdot M=L+\ell D$ and assume that $L$ is effective 
(otherwise a positive multiple is). As the sum of a non fixed effective divisor with any effective divisor is clearly not fixed, this proves the claim. 
\end{proof}

In the next lemma we use the following common convex-geometric terminology.
\begin{defin}
Let $C\subseteq \mathbb{R}^d$ be a closed convex cone, and let $P \in \mathbb{R}^d$ be a point. Let $H$ be a supporting hyperplane of $C$, and $H^{+}$ be 
the closed half-space defined by $H$ which contains $C$. We say that $H$ \emph{separates } $C$ from $P$ if $P$ is not contained in $H^{+}$.
\end{defin}

\begin{lemma}\label{lemzarde}
Let $X=X_w$ be a Bott-Samelson variety. Let $E$ be a fixed extremal divisor of the effective cone, and let $H$ be a supporting hyperplane of the nef cone 
which separates the nef cone from $E$. Let $P$ be a nef divisor on $H$. Then $D_{k\ell}:=kP+\ell E$ is a Zariski decomposition for all $k,\ell \geq 0$, i.e. $P(D_{k\ell})=kP$ and $N(D_{k\ell})=\ell E$.
\end{lemma}

\begin{proof}
	Since $kP$ is a nef subdivisor of $D_{k\ell}$, it follows from Remark \ref{remmax} that $kP\leq P(D_{k\ell})$. Since $P(D_{k\ell})$ is 
	the maximal subdivisor of $D_{k\ell}$ it is of the 
	form $kP+mE$ for some $0\leq m\leq n$. Let now $H^{+}$ be the closed half space corresponding to $H$ which contains the nef 
	cone, and $H^{-}_{<0}$:= $\mathbb{R}^n\setminus H^{+}$ its 
	complementary open half space. Since $H$ separates the nef cone from $E$, we have $E\in H^{-}$ as well as, by assumption, $P
	\in H$. It follows that $kP+mE$ lies in $H^{-}$, which means that it 
	is not nef, unless $m=0$. Therefore, the maximal nef subdivisor of $D_{k\ell}$ is $P(D_{k\ell})$.
\end{proof}

The next lemma says that all but the first extremal rays of the effective cone are indeed fixed. We use the notation from \cite{LT04} (see Section \ref{secprelim}).
\begin{lemma}
 Let $X_w$ be a Bott-Samelson variety and let $\{E_1,\ldots, E_n\}$ be the associated 
 effective basis. Then, $E_1$ is nef, and $E_2,\ldots, E_n$ are fixed.
\end{lemma}

\begin{proof}
 Since $\mathcal{O}_X(E_1)$ is the pullback of $\mathcal{O}_{\mathbb{P}^1}(1)$ by the morphism $\pi: X_w \to \mathbb{P}^1$, $E_1$ is nef. 
 
 We prove the fixedness of $E_2,\ldots, E_n$ by induction on the dimension of $X_w$. For this, let the given reduced $w$ be 
 $w=(s_{i_1},\ldots, s_{i_n})$, where the $s_{i_j}$ are simple reflections, associated to the simple roots $\alpha_{i_j}$. Then, 
 $s_{i_j} \neq s_{i_{j+1}}$ for $j=1,\ldots, n-1$, since $w$ is 
 reduced. Let $w[1]:=(s_{i_1},\ldots, s_{i_{n-1}})$, and consider the $(n-1)$-dimensional Bott-Samelson variety $X_{w[1]}$ and the 
 fibre bundle $\pi_1: X_w \to X_{w[1]}$, as well as the embedding $\iota: X_{w[1]} \hookrightarrow X_w$ with image $E_n$. 
 For any divisor $D$ on $X_{w[1]}$ we have $H^0(X_w, \pi_1^*\mathcal{O}_{X_{w[1]}}(D)) \cong H^0(X_{w[1]}, \mathcal{O}_{w[1]}(D))$. 
 From this, we can conclude by induction  that the divisors $E_2, \ldots, E_{n-1}$ are fixed. 
 
 Now, being an extremal generator of $\mbox{Eff}(X)$, if $E_n$ were not fixed, it would be nef. Hence, its restriction to 
 $X_{w[1]} \cong E_n$ would also be nef. However, \cite[Lemma 3.6.]{HY} shows that the restriction of $\mathcal{O}_{X_w}(E_n)$ to $E_n$
 has a negative $D_{n-1}$-coefficient with respect to the nef basis for $\mbox{Pic}(E_n)$. This finishes the induction step, and hence the proof.
\end{proof}

The next lemma is the key to the explicit description of the Zariski chambers. It gives a correspondence between the extremal fixed divisors and the facets of the nef cone.

\begin{lemma}\label{lemextfac}
Let $X=X_w$ be a Bott-Samelson variety. For each fixed extremal ray $E_i, i=2,\dots,n$, of the effective cone of $X$ there is a unique facet $F_i$ of the 
nef cone such that its 
supporting hyperplane $H_i$ separates the extremal ray from the nef cone.
\end{lemma}
 
\begin{proof}
Let $n=\dim X$. Then there  are $n-1$ fixed extremal rays of the effective cone and $n$ facets of the nef cone. We first claim that there is one facet, say $F_1$, such that its 
supporting hyperplane coincides with a supporting hyperplane of a facet of the effective cone. This can be seen as follows. Let us suppose that $X=X_w$ for a reduced expression $w$.
Then, $E_1,\dots,E_{n-1}$, resp. $D_1,\dots D_{n-1}$, define the effective cone, resp. the nef cone of $X_{w[1]}$. This proves that $D_1,\dots, D_{n-1}$ lie in the linear space 
defined by $E_1,\dots, E_{n-1}$. Hence, choosing $E_1,\dots, E_n$ as a basis for $\Pic_{\mathbb{R}}\cong \mathbb{R}^n$, the facet defined by $E_1,\dots, E_{n-1}$ and $D_1,\dots,D_{n-1}$ have 
the supporting hyperplane $H_1=\{x_n=0\}$. 

This facet, $F_1$, does not separate any extremal fixed divisor $E_i$ from the nef cone. 
However, it is clear that for any fixed extremal divisor $E_i$ there is at least on facet $F_j$ of the  nef cone such that its supporting hyperplane does separate the nef cone and $E_i$.
Furthermore, we claim that there are no supporting hyperplanes of the facets $F_i$ such that two distinct extremal rays $E_\ell$ and $E_k$ are separated 
from the nef cone simultaneously. 
Suppose there is, then the interior of the cones $F_i+\mbox{Cone}(E_\ell)$ and $F_i+\mbox{Cone}(E_k)$ do intersect.
But this contradicts the uniqueness of the Zariski decomposition and Lemma \ref{lemzarde}.
Altogether, we have $n-1$ extremal fixed divisors to  which we can individually associate at least one facet. But, as we have seen, one facet cannot correspond to more than one 
extremal divisor. Since there exist only $n$ facets, from which one facet does not correspond to any extremal ray, the claim is proven.
\end{proof}

We now use the above lemma to introduce the following notation.
For a Bott-Samelson variety $X_w$ of dimension $n$, we define $F_i$, $i=2,\dots, n$, as the facets of the nef cone 
which correspond to the fixed divisors $E_i$, $i=2,\dots n$, according to the above 
lemma. Furthermore, we call $F_1$ the remaining facet, which is just the facet spanned by the divisors $D_1,\dots, D_{n-1}$.
Let $H_1,\dots, H_n$ be the supporting hyperplanes corresponding to the facets $F_i$ and denote by $H_i^+$
the closed half spaces corresponding to $H_i$ which contain the nef cone.
Having the notation fixed, we are now in a position to explicitly describe the Zariski chambers of Bott-Samelson varieties.

\begin{thm}
Let $E$ be a fixed divisor with support $E_{i_1}\cup\dots \cup E_{i_\ell}$. Then $\Sigma_E$ is given by 
\begin{align*}
\Sigma_E=\Pi_E:=\left( F_{i_1}\cap \dots \cap F_{i_\ell}\right) + \mbox{Cone}(\{E_{i_1},\dots E_{i_\ell}\}).
\end{align*}
Moreover, $\Sigma_E$ defines  a Zariski chamber which is an $n$-dimensional simplex.
 \end{thm}

\begin{proof}
Let us first prove that $\Pi_E\subseteq \Sigma_E$. Choose $D$ in the relative interior of $\Pi_E$. So 
we can write $D=P+ \sum_{k=1}^\ell \lambda_k E_{i_k}$ for $P\in F_{i_1}\cap \dots \cap F_{i_\ell}$, and $\lambda_k\geq 0$ . We want to prove that this is already the 
Zariski decomposition, i.e., $P(D)=P$. Since $P$ lies in $\bigcap_{k=1,\dots,\ell} F_{i_k}$ it follows that $P+\varepsilon E_{i_k}$ for $\varepsilon>0$ and $k=1,\dots \ell$ is not nef. This proves 
the maximality of $P$ and shows $P(D)=P$ as well as $N(D)=\sum_{k=1}^\ell \lambda_k E_{i_k}$. Hence, $D\in \Sigma_E$ and by the closedness of $\Sigma_E$ the inclusion follows.

We now prove the reverse inclusion $\Sigma_E\subseteq \Pi_E$.
Let $D$ be any effective divisor such that its Zariski decomposition is given by $D=P+\sum_{k=1}^\ell \lambda_kE_{i_k}$ for $\lambda_k>0$. Then in order to 
prove that $D$ lies in $\Pi_E$ we need to show that $P\in F_{i_1}\cap \dots \cap F_{i_\ell}$.
Suppose $P$ does not lie in $F_{i_k}$ for some $k=1,\dots, \ell$. That means $P$ lies in the open half space $(H^{+}_{i_k})_{>0}:=H^{+}_{i_k}\setminus H_{i_k}$. Then for $\varepsilon>0$ small enough, we 
have $P+\varepsilon E_{i_k}$ still lies in $H^{+}_{i_k}$. By Lemma \ref{lemextfac}, $E_{i_k}$ does lie in $H^{+}_{j}$ for $j\neq k$. This proves 
that $P+\varepsilon E_{i_k}$ lies in $H^{+}_{j}$ for all $j=1,\dots,n-1$. Hence,
 $P+\varepsilon E_{i_k}$  is nef, contradicting the maximality of $P$. We have thus shown that $P\in F_{i_k}$.
Again taking the closedness of $\Pi_E$ into account, we obtain the reverse inclusion.

Let us now show that $\Sigma_E$ defines an $n$-dimensional simplex.
Let $F:=F_{i_1}\cap \dots \cap F_{i_r}$ and denote the generators of $F$ by $D_{j_1},\dots,D_{j_{n-r}}$.
If $\Sigma_{E}$ is of dimension $n$ we are done. Suppose that it is of dimension less than $n$.
That means the points $E_{i_1},\dots, E_{i_r},D_{j_1},\dots,D_{j_{n-r}}$ are linearly dependent.
Hence, there are $\lambda_i,\mu_i\geq 0$ for $i=1,\dots, n$ with $\lambda_i\neq \mu_i$ for all $i=1,\dots,n$ such that
\begin{align*}
\sum_{k=1}^{n-r}\lambda_{r+k}D_{j_k} + \sum_{k=1}^r \lambda_k E_{i_k}  =  \sum_{k=1}^{n-r}\mu_{r+k}D_{j_k}+ \sum_{k=1}^r \mu_k E_{i_k}.
\end{align*}
However, we have seen in the first part of the proof that the above decomposition is actually a Zariski decomposition.
Since this decomposition is unique we get
\begin{align*}
\sum_{k=1}^{n-r}\lambda_{r+k}D_{j_k}=\sum_{k=1}^{n-r}\mu_{r+k}D_{j_k}
\end{align*}
and 
\begin{align*}
\sum_{k=1}^r \lambda_k E_{i_k}= \sum_{k=1}^r \mu_k E_{i_k}.
\end{align*}
But both the $D_i$'s and the $E_i$'s are part of a basis of $N^1_\mathbb{R}(X_w)$. Hence, it follows that $\lambda_i=\mu_i$ for all $i=1,\dots,n$.
This contradicts the linear dependence of  $E_{i_1},\dots, E_{i_r},D_{j_1},\dots,D_{j_{n-r}}$.

\end{proof}

\subsection{Mori chambers}
 In this subsection we prove that the previously defined Zariski chambers coincide with the  Mori chambers defined in \cite{hk}.
 Let us first recall what we mean by a Mori chamber. First of all, we call two divisors $D_1$ and $D_2$ on a Mori  dream space Mori equivalent if there is an isomorphism 
 $\mbox{Proj}(R(X,D_1)) \cong \mbox{Proj}(R(X,D_2))$ such that the obvious diagram
 \begin{align*}
 \xymatrix{
 X \ar@{-->}[r]\ar@{-->}[rd]& \mbox{Proj}(R(X,D_1))\ar[d]^{\cong}\\
 & \mbox{Proj}(R(X,D_2))}
 \end{align*}
 commutes.  
Then, Mori chambers are the closure of Mori equivalence classes which have non-empty interior. 

\begin{thm}
Let $X$ be a Bott-Samelson variety. Then each Zariski chamber defines a Mori chamber and vice versa.
\end{thm}

\begin{proof}
For the Mori chamber $\mathcal{C}=\text{Nef}(X)$ this is clear.
Let us assume that two divisors $D_1$ and $D_2$ lie in the interior of a Mori chamber $\mathcal{C}\neq \text{Nef}$. Then by \cite{O16}, $D_1$ and $D_2$ are strongly Mori equivalent and in particular, their stable base loci coincide. But this just means $\text{supp}(N(D_1))=\text{supp}(N(D_2))$, which in 
turn shows that $D_1$ and $D_2$ lie in the same Zariski chamber.

Let us assume they lie in $\Sigma_E=F+\mbox{Cone}(E_{i_1},\dots,E_{i_\ell})$. By the description in \cite[ Proposition 1.11]{hk}, we know that each Mori chamber is the Minkowski sum of some $g_i^* \text{Nef}(Y_i)$, for a birational 
contraction $g_i \colon X \dashrightarrow Y_i$, and the cone generated by some extremal fixed prime divisor. However, since $\mbox{Mov}(X)=\mbox{Nef}(X)$, $g_i$ is actually 
a regular birational contraction. Thus, $g_i^*\text{Nef}(Y_i)\subset \text{Nef}(X)$, and since two Mori chambers intersect along a common face 
it follows that $g_i^*\text{Nef}(Y_i)$ actually is a face of $\text{Nef}(X)$. However, since $\mathcal{C}$ lies in $\Sigma_E$, the only way to generate a 
chamber with non-empty interior is to take $F$ as the face of the nef cone and $E_{i_1},\dots, E_{i_\ell}$ as our extremal fixed divisors.
This proves that $\mathcal{C}=\Sigma_E$.
\end{proof}

\section{Examples of Mori chamber decomposition}\label{s:ex1}
In this section we give two examples where we compute the Mori chamber decomposition of the effective cone.
Note that all necessary computations were done on a  computer, using Sage.
\subsection{A $3$-dimensional incidence variety}\label{exthreedim}
We start with the three-dimensional incidence variety $Y$ which is described in \cite[Example 2]{SS17}.
It consists of tuples of linear subspaces $(V_1,V_2,V_2^\prime)$ of $\mathbb{C}^3$ such that $V_1$ is one-dimensional and $V_2,V_2^\prime$ are two-dimensional. Furthermore the following incidences hold:
\begin{align*}
\mathbb{C}&\subseteq V_2, \quad V_1 \subseteq V_2, \quad V_1\subseteq V_2^\prime.
\end{align*}
These incidences can be illustrated in the following diagram:
\begin{align*}
\xymatrix{  
		   \mathbb{C}^3 & &\\
		   \mathbb{C}^2\ar@{-}[u] & V_2\ar@{-}[lu] & V_2^\prime\ar@{-}[llu] \\
		   \mathbb{C}\ar@{-}[u]\ar@{-}[ru] & V_1\ar@{-}[u]\ar@{-}[ru]. &
}
\end{align*}
In \cite{SS17} the relations between the divisors $E_1,E_2,E_3$ and $D_1,D_2,D_3$ were computed and are given by
\begin{align*}
D_1&=E_1 \\
D_2&=E_2+E_1\\
D_3&=E_3+E_2.
\end{align*}
Cutting the effective/nef cone with a generic hyperplane, we get the following picture:

\begin{figure}[h]
\begin{center}
\includegraphics[scale=0.8]{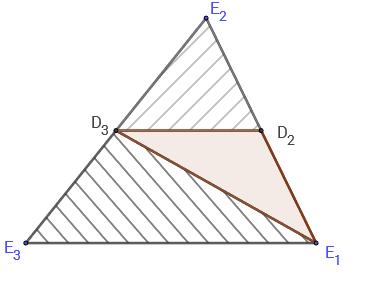}
\caption{The Mori decomposition of $Y$}
\end{center}
\end{figure}
We can see from the picture that there are exactly three Mori chambers, which are given by the nef cone, $\Nef(X)=\text{Cone}(E_1,D_2,D_3)$, and the two cones  $\text{Cone}(E_2,D_2,D_3)$ and $\text{Cone}(E_3,D_1,D_3)$.

\subsection{A $4$-dimensional incidence variety} \label{fourdim}
Let us now consider a four-dimensional incidence variety $X$ which consists of tuples of linear subspaces $(V_1,V_2,V_2^\prime,V_3)$ such that $V_1$ is one-dimensional, 
$V_2,V_2^\prime$ are two-dimensional, and $V_3$ is three-dimensional. Furthermore, the incidences are described in the following diagram:
\begin{align*}
\xymatrix{ \mathbb{C}^4 & & \\ 
		   \mathbb{C}^3\ar@{-}[u] & V_3\ar@{-}[lu] &\\
		   \mathbb{C}^2\ar@{-}[u] & V_2\ar@{-}[lu] & V_2^\prime\ar@{-}[llu]\ar@{-}[lu] \\
		   \mathbb{C}\ar@{-}[u]\ar@{-}[ru] & V_1\ar@{-}[u]\ar@{-}[ru] &.
}
\end{align*}

\noindent We define a map $q\colon X\to Y$ by  $(V_1,V_2,V_2^\prime,V_3)\mapsto(V_1,V_2,V_2^\prime)$. This makes $X$ a Bott-Samelson variety, given as a $\mathbb{P}^1$ bundle over $Y$.

We proceed by describing the new occurring divisors $E_4$ and $D_4$.
The divisor $E_4$ is the image of the embedding of $Y$ into $X$ by mapping 
\begin{align*}
(V_1,V_2,V_2^\prime)\mapsto (V_1,V_2,V_2^\prime,\mathbb{C}^3).
\end{align*}
Denote by $\mathcal{V}_1,\mathcal{V}_2, \mathcal{V}_2^\prime, \mathcal{V}_3$ the tautological vector bundles with fibres $V_1,V_2,V_2^\prime,V_3$ over the point $(V_1,V_2,V_2^\prime,V_3)\in X$.
Then $D_4$ is equal to $\det(\mathcal{V}_3)^*$.

We can describe the divisor $E_4$ as the zero set of the section
\begin{align*}
s_{E_4}\in &H^0(X,\text{Hom}(\mathbb{C}^3/\mathcal{V}_2^\prime,\mathbb{C}^4/\mathcal{V}_3))=H^0(X,(\mathbb{C}^3/\mathcal{V}_2\prime\otimes \left(\mathbb{C}^4/\mathcal{V}_3\right)^*)\\
&s_{E_4}(V_1,V_2,V_2^\prime,V_3)\colon (v+V^\prime_2)\mapsto v+V_3.
\end{align*}

Furthermore we have the following identifications, which can be readily checked:
\begin{align*}
&\mathbb{C}^4/\mathcal{V}^\prime_2\cong (\det \mathcal{V}_2^\prime)^*\\
&\mathbb{C}^4/\mathcal{V}_3\cong (\det \mathcal{V}_3)^*
\end{align*}
From this we can conclude $E_4=D_4-D_3$, which leads to
\begin{align*}
D_4=E_4+E_3+E_2.
\end{align*}
Let us now determine all the fixed faces of the effective cone. We already know from the above picture that $\text{Cone}(E_2,E_3)$ is not fixed. This implies that the face $\text{Cone}(E_2,E_3,E_4)$ is not fixed either.
The only left faces which are not known to be fixed are $\text{Cone}(E_3,E_4)$ and $\text{Cone}(E_2,E_4)$.
Let us prove that $E_2+E_4$ is fixed. Indeed, if it were not fixed, then there would exist a subdivisor $P=\lambda E_2+\mu E_4$, for $0\leq \lambda,\mu \leq 1$, which is nef.
But
\begin{align*}
 \lambda E_2+\mu E_4= \lambda (D_2-D_1)+ \mu (D_4-D_3)
\end{align*}
which is never nef as long as $\max(\lambda,\mu)>0$. But $\max(\lambda,\mu)=0$ implies  $P=0$. Thus $E_2+E_4$ is fixed, and therefore all the divisors in $\text{Cone}(E_2,E_4)$ are.
Similarly, we can prove that $\text{Cone}(E_3,E_4)$ is fixed.
This shows that $\Eff(X)$ decomposes into six Mori chambers, corresponding to the fixed divisors $E_2,E_3,E_4,E_2+E_4,E_3+E_4$, and the nef cone.

The following table displays which facets of the nef cone correspond to which extremal rays. Here, we fix the basis $(E_1,\dots, E_n)$.

\begin{figure}[h]
\begin{center}
{\def\arraystretch{2}\tabcolsep=10pt
\begin{tabular}{ccc}
\hline 
\textbf{Facet generators(Nef Cone)} & \textbf{Supp. Half-space} & \textbf{Opposite extr. ray} \\ 
\hline 
 
$D_1,D_2,D_3$& $x_4\geq 0$ & does not separate \\  
$D_1,D_2,D_4$ & $x_3-x_4\geq 0$& $E_4$ \\ 
$D_1,D_3,D_4$& $x_2-x_3\geq 0$ & $E_3$ \\ 
$D_2,D_3,D_4$& $x_1-x_2+x_3\geq 0$ & $E_2$ \\ 
\hline 
\end{tabular} }
\end{center}
\caption{Correspondence facets-extremal rays of $X$}
\end{figure}
This leads to the following notation for the facets:
\begin{align*}
F_1&=\text{Cone}(D_1,D_2,D_3)\\
F_2&=\text{Cone}(D_2,D_3,D_4)\\
F_3&=\text{Cone}(D_1,D_3,D_4)\\
F_4&=\text{Cone}(D_1,D_2,D_4).
\end{align*}
Now, we have all the necessary information to explicitly describe the Mori chambers of $X$:
\begin{figure}[h]
\begin{center}
{\def\arraystretch{2}\tabcolsep=10pt
\begin{tabular}{ccc}
\hline 
\textbf{Negative Support} & \textbf{Mori chamber}& \textbf{Color} \\ 
\hline 
$\emptyset$&  $\text{Cone}(D_1,D_2,D_3,D_4) $ &  $ {\tikzcircle[fill=blue]{5pt}}$\\ 

$E_2$& $\text{Cone}(E_2)+F_2=\text{Cone}(E_2,D_2,D_3,D_4) $& $ {\tikzcircle[fill=gray]{5pt}}$\\ 
 
$E_3$ & $\text{Cone}(E_3)+F_3=\text{Cone}(E_3,D_1,D_3,D_4)$& $ {\tikzcircle[fill=green]{5pt}}$ \\ 
 
$E_4$& $\text{Cone}(E_4)+F_4=\text{Cone}(E_4,D_1,D_2,D_4)$& $ {\tikzcircle[fill=magenta]{5pt}}$\\ 
 
$E_2\cup E_4$& $\text{Cone}(E_2,E_4)+F_2\cap F_4=\text{Cone}(E_2,E_4,D_2,D_4)$& $ {\tikzcircle[fill=plum]{5pt}}$\\ 
 
$E_3\cup E_4$ & $\text{Cone}(E_3,E_4)+F_3\cap F_4=\text{Cone}(E_3,E_4,D_1,D_4)$& $ {\tikzcircle[fill=brown]{5pt}}$ \\ 
\hline 
\end{tabular} 
}
\end{center}
\caption{Mori chambers of $X$}
\end{figure}
\newpage
Finally, let us illustrate the Mori decomposition of $X$ by plotting a slice of the chamber decomposition with a generic hyperplane.
For a better overview we display the decomposition from two different perspectives.

\begin{figure}[h]
\begin{center}
\includegraphics[scale=0.8]{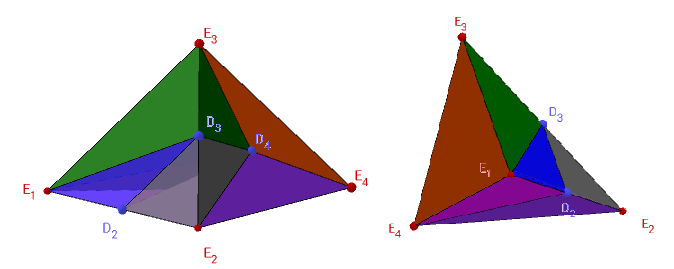}
\end{center}
\caption{Mori chamber decomposition of $X$}
\end{figure}

\section{Global Newton-Okounkov bodies on Bott-Samelson varieties and the global value semigroup}

In this section, we consider (global) Newton-Okounkov bodies with respect to the so-called `horizontal' flag. Furthermore, we show that the global semigroup $\Gamma_{Y_\bullet}(X)$ is finitely generated.
\subsection{The horizontal flag}
We  start with a description of the \emph{horizontal flag}.
Let $w=(s_1,\dots, s_n)$ be a reduced word and denote by $X_w$ the corresponding Bott-Samelson variety. Let furthermore $E_1,\dots, E_n$ be the effective basis, satisfying $E_i\cong X_{w(i)}$ for $w(i)=(s_1,\dots,\hat{s}_i,\dots,s_n)$.
Moreover, let us define, for $i=1,\dots, n$, the truncated sequence $w[i]=(s_1,\dots,s_{n-i})$. 
Then, if $w$ is a reduced sequence, the sequence $w[i]$ is still reduced and $X_{w[i]}\subseteq X_w$ is a closed subvariety of codimension $i$ which is again a Bott-Samelson variety. We can also wr
ite $X_{w[i]}=E_n\cap \dots \cap E_{n-i+1}$ and represent $X_{w[i]}$ as a closed subvariety of $X_w$ as follows
\begin{align*}
X_{w[i]}=\{[(p_1,\dots,p_n)]\in X_w \ \vert \ p_n=\dots=p_{n-i+1}=e \}.
\end{align*}
We define the \emph{horizontal flag} $Y_\bullet$ as follows.
For $i=1,\ldots, n$, we set
\begin{align*}
Y_{i}:=X_{w[i]}.
\end{align*}
We write ${Y_k}_\bullet=\left(Y_k\supseteq Y_{k+1}\supseteq \dots \supseteq Y_n \right)$
for the induced flag on $Y_k$.
Note that the effective basis of $Y_i$ is given by $(E_1)_{\vert Y_i},\dots (E_{n-i})_{\vert Y_{n-i}}$, and the $\mathcal{O}(1)$-basis of $Y_i$ is given by
$(D_1)_{\vert Y_i},\dots (D_{n-i})_{\vert Y_{n-i}}$. For the sake of simplifying the notation, we shall omit the restrictions and simply write $E_1,\dots, E_{n-i}$ and $D_1,\dots, D_{n-i}$ for divisors 
on $Y_i$ whenever no confusion should arise.

\begin{remark}
Note that Newton-Okounkov bodies with respect to the horizontal flag were already studied in \cite{HY}. In the mentioned article they show the finite generation of the semigroup 
$\Gamma(X_w,D)$ under a condition which they called condition $\textbf{(P)}$. In the following, we shall generalize this result to all divisors $D$ on $X_w$.
It is also worth to note that the techniques used in the above mentioned article substantially differ from ours. While their approach relied on representation theory and combinatorics, we mainly use our earlier 
established Mori-theoretic properties of Bott-Samelson varieties and results from \cite{LT04}.
\end{remark}

\subsection{Rational polyhedrality of global Newton-Okounkov bodies}

The key to proving the finite generation of the semigroup, as well as the rational polyhedrality of the global Newton-Okounkov body, is the following lemma.

\begin{lemma}\label{lemcohomology}
Let $D$ be a nef divisor on $X_w$. Then the restriction map
\begin{align*}
H^0(X,\mathcal{O}_X(D))\to H^0(E_n,\mathcal{O}_{E_n}(D))
\end{align*}
is surjective.
\end{lemma}
\begin{proof}
In \cite[Thm 7.4]{LT04} it was  proved in particular that for a nef divisor $D$, the first cohomology group $H^1(X,\mathcal{O}(D-E_n))$ vanishes. This shows that the restriction morphism:
\begin{align*}
H^0(X,\mathcal{O}_X(D))\to H^0(E_n,\mathcal{O}_{E_n}(D))
\end{align*} 
is surjective if $D$ is nef.
\end{proof}
\begin{thm}\label{thmglobal}
Let $X=X_w$ be a Bott-Samelson variety and $Y_\bullet$ be the horizontal flag.
Then the global Newton-Okounkov body $\Delta_{Y_\bullet}(X_w)$ is rational polyhedral.
\end{thm}
\begin{proof}
The proof is based on results already established in \cite{SS17}.
Recall that the divisors $D_1,\dots, D_n$ form the $\mathcal{O}(1)$-basis on $X_w$. Denote by $\Gamma(D_1,\dots,D_n)$ the semigroup generated by $D_1,\dots,D_n$ in $N^1(X)$. 
We define the semigroup 
\begin{align*}
S(D_1,\dots,D_n):=\{(\nu_{Y_\bullet}(s),D)\in \mathbb{N}^n\times \Gamma(D_1,\dots,D_n) \ \vert & \\
s\in H^0(X_w,\mathcal{O}(D)), \nu_1(s)=0\}&.
\end{align*}
It follows from \cite[Theorem 3.3]{SS17} that $\Delta_{Y_\bullet}(X_w)$ is rational polyhedral if $\text{Cone}(S(D_1,\dots,D_n))$ is rational polyhedral.
Consider now the semigroup
\begin{align*}
S_1(D_1,\dots,D_n):=\{(\nu_{Y_1}(s),D)\in \mathbb{N}^{n-1}\times \Gamma({D_1}\vert_{Y_1},\dots,{D_n}\vert_{Y_1}) \vert & \\
s\in H^0(Y_1,\mathcal{O}_{Y_1}(D))\}&.
\end{align*}
We define a natural map
\begin{align*}
q_0: S(D_1,\dots, D_n) &\rightarrow S_1(D_1,\dots, D_n)\\
 (\nu(s), D)&\mapsto ((\nu_2(s),\dots,\nu_n(s)), D  \cdot Y_1),
\end{align*}
which extends to the linear map
\begin{align}
q: \qquad \R^n \oplus N^1(X_w)_\R &\to \R^{n-1} \oplus N^1(Y_1)_\R, \label{E: linmapq}\\
((x_1,\dots,x_n), D) &\mapsto ((x_2,\dots, x_n), D \cdot Y_1). \nonumber
\end{align}

\noindent We now use the following fact which we shall prove in the below lemma: 
\begin{align*}
	&\overline{\text{Cone}(S(D_1,\dots,D_n))}= \\  
	&q^{-1}(\overline{\text{Cone}(S_1(D_1,\dots,D_n))}) \cap (\{0\} \times \R^{n-1}_{\geq 0} 
	\times \text{Cone}(D_1,\ldots, D_n)).  \nonumber
	\end{align*}

This identity shows that $\overline{\text{Cone}(S(D_1,\dots,D_n))}$ is rational polyhedral if $\overline{\text{Cone}(S_1(D_1,\dots,D_n))}$ is rational polyhedral.

Now, we proceed by induction on the dimension $n$ of $X_w$.
If $n=1$, then $X_w\cong\mathbb{P}^1$. It can be easily checked that the global Newton-Okounkov body of $\mathbb{P}^1$ with respect to any admissible flag is rational polyhedral.

Assume now that the assertion is true for $n-1$. Then, $\Delta_{{Y_1}_\bullet}(Y_1)$ is rational polyhedral, and we have
\begin{align*} 
\overline{\text{Cone}(S_1(D_1,\dots,D_n))} =pr_2^{-1}\left(\overline{\text{Cone}({D_1}_{\vert Y_1},\dots,{D_n}_{\vert Y_1})} \right)\cap \Delta_{{Y_1}_\bullet}(Y_1) .
\end{align*}
But this implies that $\overline{\text{Cone}(S_1(D_1,\dots,D_n))}$ and $\overline{\text{Cone}(S(D_1,\dots,D_n))}$ are rational polyhedral.
Finally, this proves that $\Delta_{Y_\bullet}(X_w)$ is rational polyhedral. 
\end{proof}
\begin{lemma}
With the notation introduced above, we have
\begin{align*}
	&\overline{\text{Cone}(S(D_1,\dots,D_n))}= \\  
	&q^{-1}(\overline{\text{Cone}(S_1(D_1,\dots,D_n))}) \cap (\{0\} \times \R^{n-1}_{\geq 0} 
	\times \text{Cone}(D_1,\ldots, D_n)).  \nonumber
	\end{align*}
\end{lemma}
\begin{proof}
In order to show the inclusion `$\subseteq$', it is enough to show this inclusion for the semigroup $S(D_1,\dots,D_n)$ since both sides are closed convex cones. 
So, let $(0,a_2,\dots,a_n,D)\in  S(D_1,\dots,D_n)$. Then this is clearly a preimage of $(a_2,\dots,a_n,D \cdot Y_1)$ under $q$ and lies in $(\{0\} \times \R^{n-1}_{\geq 0} \times \text{Cone}(D_1,\ldots, D_n))$. 
This shows the first inclusion. 

For the second inclusion `$\supseteq$', note that both sides are closed sets. Hence, it is enough to show that the inclusion holds for rational points 
in the interior, and--since both sides are also cones--it is enough to 
show the inclusion for integral points in the interior.
Let therefore
\begin{align*}
(0,a,D)\in q^{-1}(\overline{\text{Cone}(S_1(D_1,\dots,D_n))}) \cap (\{0\} \times \R^{n-1}_{\geq 0} 
	\times \text{Cone}(D_1,\ldots, D_n))
\end{align*}	
be an integral point in the interior. By definition, 
\begin{align*}
(a,D|Y_1)\in \text{Cone}(S_1(D_1,\dots,D_n)).
\end{align*}
 After scaling appropriately, we can assume $(ka,kD \cdot Y_1)\in S_1(D_1,\dots, D_n)$.
This means that there is a section $s\in H^0(Y_1,\mathcal{O}_{Y_1}(kD))$ such that $\nu_{Y_{1\bullet}}(s)=ka$.
By Lemma \ref{lemcohomology}, we can lift the section $s$ to a section $\tilde{s}\in H^0(X,\mathcal{O}_X(kD))$ such that $\tilde{s}|_{Y_1}=s$. Then, we clearly have 
$\nu_{Y_\bullet}(\tilde{s})=(0,ka)$. This proves that $(0,ka,kD)\in S(D_1,\dots,D_n))$, which implies
$(0,a,D)\in \overline{\text{Cone}(S(D_1,\dots,D_n))}$.

\end{proof}

\subsection{Value semigroups}
We need the following notation: for a fixed flag  $Y_\bullet$ on $X$ and an effective divisor $D\neq 0$, we define the following semigroup
\begin{align*}
\Gamma_{Y_\bullet}(D)=\bigsqcup_{k\in \mathbb{N}}\Gamma_{Y_\bullet}(D)_k:= \{(\nu_{Y_\bullet}(s),k)\ \vert  \ k\in \mathbb{N}, \ 
s\in H^0(X,\mathcal{O}_X(kD))\setminus \{0\} \}
\end{align*}
as well as
\begin{align*}
\Gamma_{Y_\bullet}(D)_{\nu_1=a}=\bigsqcup_{k\in\mathbb{N}}\left(\Gamma_{Y_\bullet}(D)_{\nu_1=a}\right) _k=\{(\nu_1(s),\dots,\nu_n(s),k) \ \vert \ k\in\mathbb{N}&\\
 \ s\in H^0(X,\mathcal{O}_X(kD)), \ \nu_1(s)=ak \}.
\end{align*}
If $D$ is a $\mathbb{Q}$-divisor such that $kD$ is integral for a given $k\in \mathbb{N}$, we define 
$$\Gamma_{Y_\bullet}(D)_k=\{(\nu(s),k) \ \vert \ s\in H^0(X,\mathcal{O}_X(kD)).$$
Furthermore, for $a>0$ we abbreviate $$D_a:=D-aY_1,$$ as well as 
\begin{align}
P_a:=P(D_a) \quad \mbox{and} \quad  N_a:=N(D_a) \label{E: Zariskiofdiff}
\end{align}

\begin{prop}\label{propsemigroup}
Let $X=X_w$ be an $n$-dimensional Bott-Samelson variety. Let $Y_\bullet$ be an admissible flag with $Y_1=X_{w[1]}=E_n$. Let $D$ be an effective divisor. 
Then, the identity
\begin{align*}
\left(\Gamma_{Y_\bullet}(D)_{\nu_1=a}\right)_k=\{ak\}\times \left( \Gamma_{{Y_1}_\bullet}({P_a}\vert_{Y_1})_k + k\cdot \nu_{{Y_1}_\bullet}({N_a}\vert_{Y_1})\right)
\end{align*}
holds for all $a\in \mathbb{Q}$ such that $\Gamma_{Y_\bullet}(D)_{\nu_1=a}\neq \emptyset$, and $k> 0$ such that $kP_a$ and $kN_a$ are integral.
\end{prop}
\begin{proof}
 We claim that $N_a$ does not contain $Y_1$ in its support. Suppose it does, then
\begin{align*}
D=P_a+(N_a+ tY_1)
\end{align*}
is the Zariski decomposition of $D$, i.e. $N(D)=N_a+tY_1$. However, this proves $\nu_1(s)>ak$ for each $s\in H^0(X,\mathcal{O}_X(kD))$, which contradicts the fact 
that $\Gamma_{Y_\bullet}(D)_{\nu_1=a}\neq \emptyset$.

For each $a\in\mathbb{Q}$, we choose $k\in \mathbb{N}$ such that $kP_a$ and $kN_a$ are integral.
Let us now show the inclusion `$\subseteq$'. Let $s\in H^0(X,\mathcal{O}_X(kD))\setminus\{0\}$ such that $\nu_1(s)=ak$. Then consider $\tilde{s}:=s/s_{Y_1}^{ka}$.
We can decompose this section $\tilde{s}=s_{P}\otimes s_{N}$.
By construction,
\begin{align}\label{eqval}
\nu_{Y_\bullet}(s)=(ak,\nu_{{Y_1}_\bullet}(\tilde{s}\vert_{Y_1}))=(ak,\nu_{{Y_1}_\bullet}({s_{P}}\vert_{Y_1})+k\cdot \nu_{{Y_1}_\bullet}({N_a}\vert_{Y_1})).
\end{align}
This proves the desired inclusion.

Now, in order to prove the reverse inclusion `$\supseteq$', consider an arbitrary section $\tilde{s}_{P}\in H^0(Y_1,\mathcal{O}_{Y_1}(kP_a))$. By Lemma \ref{lemcohomology}, there is a 
section $s_P \in H^0(X,\mathcal{O}_X(kP_a))$ which coincides with $\tilde{s}_{P}$ on $Y_1$.
Define $\tilde{s}:= s_P\otimes s_N$, where $s_N$ is a section of $\mathcal{O}(kN_a)$, and finally, set $s:=\tilde{s}\otimes s_{Y_1}^{ak}\in H^0(X,\mathcal{O}_X(kD))$.
Then the reverse inclusion follows from equation \eqref{eqval}.

\end{proof}
\begin{cor}\label{corbig}
Let $D$ be a big divisor on $X_w$, and $Y_\bullet$ be the horizontal flag.
Then for every $a\in\mathbb{Q}$ such that $D-aY_1$ is big and $(D-aY_1)\vert_{Y_1}$ is big, we have
\begin{align*}
\Delta_{Y_\bullet}(D)_{\nu_1=a}=\Delta_{{Y_1}_\bullet}({P_a}\vert_{Y_1})+\nu_{{Y_1}_\bullet}({N_a}\vert_{Y_1}).
\end{align*}
\end{cor}
\begin{proof}
If $a>0$ then it follows from the fact that $D-aY_1$ is big that $\{a\}\times \mathbb{R}^{n-1}$ meets the interior of $\Delta_{Y_\bullet}(D)$. But then we have 
\begin{align*}
\text{Cone}(\Gamma_{Y_\bullet}(D)_{\nu_1=a})=\text{Cone}(\Gamma_{Y_\bullet}(D))_{\nu_1=a}
\end{align*}
which was established in \cite[(4.8)]{LM09} and proves the proposition for $a>0$.
For $a=0$, replace $D$ by $D+\varepsilon A$, for an ample class $A$, and let $\varepsilon \to 0$.
\end{proof}

\subsection{Numerical and valuative Newton-Okounkov bodies}
In this subsection we consider Newton-Okounkov bodies of effective (not necessarily big) divisors.

\begin{lemma}
Let $X_w$ be an $n$-dimensional Bott-Samelson variety. Let $Y_\bullet$ be the horizontal flag. Let $D$ be an effective divisor. 
Then, the inclusion
\begin{align*}
\Delta^{val}_{Y_\bullet}(D)_{\nu_1=a}\supseteq \{a\}\times \left( \Delta^{val}_{{Y_1}_\bullet}({P_a}\vert_{Y_1})+ \nu_{{Y_1}_\bullet}({N_{a}}\vert_{Y_1})\right)
\end{align*}
holds for all $a\in \mathbb{Q}$  such that $\Delta^{val}_{Y_\bullet}(D)_{\nu_1=a}\neq \emptyset$.
\end{lemma}
\begin{proof}
The above inclusion follows from Proposition \ref{propsemigroup} and the fact that $\text{Cone}(\Gamma_{Y_\bullet}(D)_{\nu_1=a})\subseteq\text{Cone}(\Gamma_{Y_\bullet}(D))_{\nu_1=a}$ for all rational $a\geq 0$.

\end{proof}

\begin{lemma}
Let $D$ be an effective divisor on $X_w$. Let $Y_\bullet$ be the horizontal flag.
Then,
\begin{align*}
\Delta_{Y_\bullet}^{num}(D)_{\nu_1=a}\subseteq \{a\}\times \left( \Delta_{{Y_1}_\bullet}^{num}({P_a}\vert_{Y_1})+\nu_{{Y_1}_\bullet}({N_a}\vert_{Y_1})\right)
\end{align*}
for all $a\in \mathbb{Q}$  such that $\Delta^{num}_{Y_\bullet}(D)_{\nu_1=a}\neq \emptyset$.
\end{lemma}
\begin{proof}
We now fix an ample divisor $A$.
Note that it follows from the assumption on $a$ that $D-aY_1$ is effective, and that $Y_1$ is not contained in the support of the negative part of $D$. Therefore, the 
divisor $D-aY_1+1/k\cdot A$, as well as its restriction to $Y_1$, is big. 

Now, we can use Corollary \ref{corbig} to deduce that
\begin{align*}
\Delta_{Y_\bullet}^{num}(D)_{\nu_1=a}&=\left(\bigcap_{k\in \mathbb{N}} \Delta_{Y_\bullet}(D+1/k\cdot A)\right) \cap \left( \{a\}\times \mathbb{R}^{n-1}\right)\\
&=\bigcap_{k\in\mathbb{N}}\left((\Delta_{Y_\bullet}(D+1/k \cdot A)\cap (\{a\}\times \mathbb{R}^{n-1}) )\right)\\
&=\{a\}\times \left( \bigcap_{k\in \mathbb{N}}( \Delta_{{Y_1}_\bullet}({P_{a,1/k}}\vert_{Y_1})+\nu_{{Y_1}_\bullet}({N_{a,1/k}}\vert_{Y_1}))\right).
\end{align*}
where $P_{a,b}:=P(D+bA -a Y_1)$ and $N_{a,b}:=N(D+bA -a Y_1)$.

Let us now show that $\nu_{Y_\bullet}(N_{a,1/k}) \to \nu_{Y_\bullet}(N_a)$ as $k \to \infty$.
Indeed, there is an integer $K>0$ such that $K\cdot D_a+ A$ and $D$ lie in the same Zariski chamber.
Then, for each $k>K$, we have
\begin{align*}
N(kD_a+A)=N((k-K)D_a+(KD_a+A))=(k-K)N(D_a)+ N(KD_a+A).
\end{align*}
Dividing by $k$ and applying the valuation $\nu_{Y_\bullet}$ yields the result.

\noindent The fact that \begin{align*}
\bigcap_{k\in \mathbb{N}}\left(\Delta_{{Y_1}_\bullet}({P_{a,1/k}}\vert_{Y_1})\right)\subseteq \Delta^{num}_{{Y_1}_\bullet}({P_a}\vert_{Y_1}).
\end{align*}
follows by observing that ${P_{a,1/k}}\vert_{Y_1}$ converges to ${P_a}\vert_{Y_1}$ as $k\to \infty$.
Hence, the result follows. 
\end{proof}

\begin{thm}
Let $D$ be an effective divisor on a Bott-Samelson variety $X=X_w$, and $Y_\bullet$ the horizontal flag.
Then
\begin{align*}
\Delta^{num}_{Y_\bullet}(D)=\Delta^{val}_{Y_\bullet}(D).
\end{align*}
Moreover, we have 
\begin{align*}
\Delta_{Y_\bullet}(D)_{\nu_1=a}=\{a\}\times \left( \Delta_{{Y_1}_\bullet}({P_a}\vert_{Y_1})+\nu_{{Y_1}_\bullet}({N_a}\vert_{Y_1})\right)
\end{align*}
for all $a\in \mathbb{Q}$  such that $\Delta_{Y_\bullet}(D)_{\nu_1=a}\neq \emptyset$.
\end{thm}
\begin{proof}

We prove this by induction over the dimension $n$ of $X_w$. For $n=1$ this is trivial since on the curve $X_w\cong \mathbb{P}^1$, the only effective non-big divisor is the zero divisor. 

Let us now suppose that equality holds for Bott-Samelson varieties of dimension $n-1$.
Let $D$ be an effective divisor on $X_w$ and $A$ and arbitrary ample class. Let $a\in \mathbb{Q}$ be such that $\Gamma_X(D)_{\nu_1=a}\neq \emptyset$.
Then we have 
\begin{align*}
\Delta_{Y_\bullet}^{num}(D)_{\nu_1=a} &\subseteq\{a\}\times \left( \Delta_{{Y_1}_\bullet}^{num}({P_a}\vert_{Y_1})+\nu_{{Y_1}_\bullet}({N_a}\vert_{Y_1})\right)\\
&=\{a\}\times \left( \Delta_{{Y_1}_\bullet}^{val}({P_a}\vert_{Y_1})+\nu_{{Y_1}_\bullet}({N_a}\vert_{Y_1})\right)\\
&\subseteq\Delta_{Y_\bullet}^{val}(D)_{\nu_1=a}.
\end{align*}
This proves $\Delta_{Y_\bullet}^{num}(D)\subseteq \Delta_{Y_\bullet}^{val}(D)$.
Since the reverse inclusion is always true, this proves the claim.
\end{proof}
\begin{cor}\label{corvalpoints}
Let $D$ be an effective divisor on $X_w$. Let $Y_\bullet$ be the horizontal flag and $b=(b_1,\dots,b_n)\in \Delta_{Y_\bullet}(D)$ be a
rational point.
Then there is an integer $k\in\mathbb{N}$ such that $k\cdot b \in \Gamma_{Y_\bullet}(D)_k$.
\end{cor}
\begin{proof}
We prove this by induction on the dimension $n$ of $X_w$. If $n=1$, this means that $X_w=\mathbb{P}^1$, the statement is easy to check.

Let now $X_w$ be of dimension $n$ and suppose $(b_1,\dots,b_n)\in \Delta_{Y_\bullet}(D)$. Hence, 
\begin{align*}
(b_1,b_2,\dots,b_n)\in \Delta_{Y_\bullet}(D)_{\nu_1=b_1}=\{b_1\}\times \left( \Delta_{{Y_1}_\bullet}^{num/val}({P_{b_1}}\vert_{Y_1})+\nu_{{Y_1}_\bullet}({N_{b_1}}\vert_{Y_1})\right).
\end{align*}
Now, we can use the induction hypothesis to deduce that 
\begin{align*}
k\cdot(b_1,\dots,b_n)\in \{b_1k\}\times \left( \left(\Gamma_{{Y_1}_\bullet}({P_{b_1}}\vert_{Y_1})\right)_k + k\cdot \nu_{{Y_1}_\bullet}(({N_{b_1}}\vert_{Y_1})\right).
\end{align*}
But by Proposition \ref{propsemigroup}, we have 
\begin{align*}
k\cdot(b_1,\dots,b_n)\in \left(\Gamma_{Y_\bullet}(D)_{\nu_1=b_1}\right)_k\subseteq \Gamma_{Y_\bullet}(D).
\end{align*}
This proves the claim.
\end{proof}

\begin{thm}\label{thmglobalsem}
Let $X=X_w$ be a Bott-Samelson variety, $Y_\bullet$ the horizontal flag.
Then the global semigroup
\begin{align*}
\Gamma_{Y_\bullet}(X_w):= \{ (\nu_{Y_\bullet}(s), D) \ \vert \ D \in \Pic(X_w), s \in H^0(X_w,\mathcal{O}_{X_w}(D)) \}  
\end{align*}
is  finitely generated.
\end{thm}
\begin{proof}
It was proved in Theorem \ref{thmglobal} that $\Delta_{Y_\bullet}(X_w)=\overline{\text{Cone}(\Gamma_{Y_\bullet}(X_w))}$ is rational polyhedral. First note that $\text{Cone}(\Gamma_{Y_\bullet}(X_w))$ is 
already closed: let  $(a,D)\in \Delta_{Y_\bullet}(X_w)$ be a rational point. This means that $a\in \Delta^{num}_{Y_\bullet}(D)$. Then, by Corollary \ref{corvalpoints} we can deduce that $ka\in \Gamma_{Y_\bullet}(D)_k$. 
But this clearly proves that $(ka,kD)\in \Gamma_{Y_\bullet}(X_w)$, and therefore
$(a,D)\in \text{Cone}(\Gamma_{Y_\bullet}(X_w))$.



It follows then from \cite[Corollary 2.10]{BG09} that $\Gamma_{Y_\bullet}(X)$ is finitely generated.

\end{proof}

We end this section with a sufficient condition for  the global semigroup $\Gamma_{Y_\bullet}(X)$ to be  normal, namely with the existence of an integral Zariski decomposition.
\begin{defin}
We say that $X$ admits  \emph{integral Zariski decompositions} if it admits a Zariski decomposition, and both the divisors $N(D)$ and $P(D)$ of an integral divisor $D$ are integral.
\end{defin}
\begin{prop} \label{propnormal}
Let $X_w$ be a Bott-Samelson variety that admits integral Zariski decompositions. Let  $Y_\bullet$ be the horizontal flag. Then $\Gamma_{Y_\bullet}(X_w)$ is a normal semigroup.
\end{prop}
\begin{proof}
First of all, we prove that for an effective divisor $D$ the semigroup $\Gamma_{Y_\bullet}(D)$ is normal.
We proceed by induction on the dimension. 

Let $\dim X_w=1$. Then $X_w\cong \mathbb{P}^1$, and it is not difficult to see that $\Gamma_{P}(D)$ is a normal semigroup for every point $P\in X_w$.
Now let us suppose the claim holds in dimension $n-1$.
In order to use the induction hypothesis we need to prove that $Y_1$ admits integral Zariski decompositions.
Let $D$ be an effective divisor on $Y_1$. Let $\iota \colon Y_1\to X_w$ be the closed embedding of $Y_1$ into $X_w$, and
let $\pi\colon X_w\to Y_1=X_{w[1]}$ be the projection to the first $n-1$ coordinates. It then follows that $\text{id}=\pi\circ \iota$. We claim that
$D=\iota^*P(\pi^*D)+\iota^*N(\pi^*D)$ is the Zariski decomposition of $D$. This proves that $Y_1$ admits integral Zariski decompositions.

Since $\iota^*P(\pi^*D)$ is nef and $\iota^*N(\pi^*D)$ is effective, we have $P(D)\geq \iota^*P(\pi^*D)$.
Consider now $\pi^*D=\pi^*P(D)+\pi^*N(D)$. Again, since $\pi^*P(D)$ is nef and $\pi^*N(D)$ is effective, we conclude
$\pi^*P(D)\leq P(\pi^*D)$. Applying $\iota^*$ on both sides, we get $P(D)\leq \iota^*P(\pi^*D)$. This proves the claim
 
Now suppose that $((ma_1,\dots,ma_n),mk)\in \left(\Gamma_{X}(D)\right)_{mk}$ for a tuple of nonnegative integers $(a_1,\dots,a_n,k)\in\mathbb{N}^{n+1}$.
The divisor $mkD_{a_1/k}$ is integral. Since $X_w$ admits integral Zariski decompositions, the divisors $mkP_{a_1/k}$ and $mkN_{a_1/k}$ are integral as well. 
We can thus use Proposition \ref{propsemigroup} to deduce that 
\begin{align*}
((ma_2,\dots,ma_n), mk)\in \Gamma_{{Y_1}_\bullet }({P_{a_1/k}}\vert_{Y_1})_{mk}+mk\cdot \nu_{{Y_1}_\bullet}({N_{a_1/k}}\vert_{Y_1}).
\end{align*}
Put $(b_2,\dots,b_n):=k\cdot \nu_{{Y_1}_\bullet}({N_{a_1/k}}\vert_{Y_1})\in \mathbb{Z}^{n-1}$.
Then, 
\begin{align*}
m\cdot(a_2-b_2,\dots, a_n-b_n, k)\in \Gamma_{Y_1}({P_{a_1/k}}\vert_{Y_1})_{mk}.
\end{align*}
Hence, we can use the induction hypothesis for $Y_1$ and $k{P_{a_1/k}}\vert_{Y_1}$ to conclude that 
\begin{align*}
(a_2-b_2,\dots,a_n-b_n,k)\in\Gamma_{Y_1}({P_{a_1/k}}\vert_{Y_1})_k.
\end{align*}
We again use the fact that $X_w$ induces integral Zariski decompositions to deduce that $kP_{a_1/k}$ and $kN_{a_1/k}$ are integral.
Hence, we can use Proposition \ref{propsemigroup}  to deduce that $(a_1,\dots,a_n,k)\in  \Gamma_{Y_\bullet}(D)_{k}$.
This proves that $\Gamma_{Y_\bullet}(D)$ is a normal semigroup. But then it follows easily that $\Gamma_{Y_\bullet}(X_w)$ is normal.
\end{proof}

\subsection{Connection between the global semigroup and the Cox ring}
The finite generation of the global semigroup is connected to the finite generation of the Cox ring.
More precisely we have the following theorem.
\begin{thm}\label{thmcox}
Let $X$ be a  $\mathbb{Q}$-factorial variety with $N^1(X)=\Pic(X)$. Let $Y_\bullet$ be an admissible flag.
Suppose $\Gamma_{Y_\bullet}(X)$ is finitely generated by 
\begin{align*}
(\nu_{Y_\bullet}(s_1),D_1),\ldots, (\nu_{Y_\bullet}(s_N),D_n).
\end{align*}
 Then $X$ is a Mori dream space, and the Cox ring $\operatorname{Cox}(X)$ is generated by the sections $s_1,\dots, s_N$.
\end{thm}
\begin{proof}
Let $R$ be the $\mathbb{C}$-algebra which is generated by the sections $s_1,\dots,s_N$.
Let $D$ be any effective divisor in $X$. Let $$k:=h^0(X,\mathcal{O}_X(D))=\vert \nu_{Y_\bullet}(H^0(X,\mathcal{O}_X(D))\setminus\{0\})\vert.$$
Since the $(\nu_{Y_\bullet}(s_1),D_1),\dots (\nu_{Y_\bullet}(s_N),D_n)$ generate $\Gamma_{Y_\bullet}(X)$, it follows that there are 
sections $f_1,\dots,f_k\in R\cap H^0(X,\mathcal{O}_X(D))\setminus\{0\}$ with distinct values, and it then follows from \cite[Proposition 2.3]{KK12}
that $f_1,\dots, f_k$ are linearly independent. This proves that they yield a basis form $H^0(X,\mathcal{O}_X(D))$ and that every section $s\in H^0(X,\mathcal{O}_X(D))$ belongs to the algebra $R$. This 
proves that $R\cong \operatorname{Cox}(D)$.
\end{proof}
In particular, the above theorem shows that $\Gamma_{Y_\bullet}(X)$ cannot be finitely generated unless $X$ is a Mori dream space.

\subsection{Newton-Okounkov bodies of Schubert varieties}
We can now use our results on Bott-Samelson varieties to deduce some consequences for Schubert varieties.

Let $P \subseteq G$ be any parabolic subgroup containing $B$, and let $w=(s_1,\ldots, s_n)$ be a reduced expression for which there is a birational morphism 
\begin{align*}
p\colon X_w\to Z_{\overline{w}}
\end{align*}
with $Z_{\overline{w}}$ denoting the Schubert variety corresponding to $\overline{w}:=s_1\cdots s_n$ in the partial flag variety $G/P$. In 
particular, one special case is $Z_{\overline{w}}=G/P$.

As $Z_{\overline{w}}$ is normal, for every effective divisor $D$ on $Z_{\overline{w}}$ we have 
$$
	H^0(Z_{\overline{w}},\mathcal{O}_{Z_{\overline{w}}}(D))\cong H^0(X_{w},p^*\mathcal{O}
	_{Z_{\overline{w}}}(D)).
$$
Hence, we can use the horizontal flag on $X_w$ to define a valuation
\begin{align*}
\nu_{Y_\bullet}\colon \bigsqcup_{D \in \mbox{Pic}(Z_{\overline{w}})} H^0(Z_{\overline{w}},\mathcal{O}_{Z_{\overline{w}}}(D))\setminus\{0\}\to \mathbb{N}^n
\end{align*}
and a corresponding (global) Newton-Okounkov body $\Delta_{Y_\bullet}(D)$ (resp. $\Delta_{Y_\bullet}(Z_{\overline{w}})$).

We can now use our previous findings to deduce the following.

\begin{thm}
Let $Z_{\overline{w}} \subseteq G/P$ be the Schubert variety for the reduced word $w$. Let $\nu_{Y_\bullet}$ be the above described 
valuation-like function on $Z_{\overline{w}}$. Then the global semigroup $\Gamma_{Y_\bullet}(Z_{\overline{w}})$ is finitely generated. 
Hence, $\Delta_{Y_\bullet}(Z_{\overline{w}})$ is rational polyhedral.\\
\noindent In particular, the global semigroup $\Gamma_{Y_\bullet}(G/P)$ for any partial flag variety $G/P$ is finitely generated.
\end{thm}
\begin{proof}
Let $D_1,\dots, D_k$ be the generators of the effective cone $\operatorname{Eff}(Z_{\overline{w}})$. 
Then $\Delta_{Y_\bullet}(Z_{\overline{w}})=\Delta_{Y_\bullet}(X_w) \cap (\mathbb{R}^n\times \text{Cone}(p^*D_1,\dots,p^*D_n))$ is clearly rational polyhedral.
Furthermore, we have
\begin{align*}
\Gamma_{Y_\bullet}(Z_{\overline{w}})=\{(\nu_{Y_\bullet}(s),[D]) \ \vert \ D\in \Gamma(p^*D_1,\dots,p^*D_k), \ s\in H^0(X_w,\mathcal{O}_{X_w}(D)) \}.
\end{align*}

It follows then, completely analogously to the proof of Theorem \ref{thmglobalsem}, that $\Gamma_{Y_\bullet}(Z_{\overline{w}})$ is finitely generated.
\end{proof}

\begin{rem}
 In \cite{FFL}, Feigin, Fourier, and Littelmann show that partial flag varieties $G/P$ for the groups $SL_n, Sp_n$, and $G_2$ admit rational polyhedral 
 local Okounkov bodies with respect to a valuation defined in local coordinates. The global description of their valuation seems to us to amount to considering 
 the blow-up $Bl_x(X_w)$ at a point $x \in X_w$ of a Bott-Samelson resolution $X_w$ of $G/P$ and choosing a suitable linear flag in the projective space 
 $\mathbb{P}(T_x(X_w))$.
\end{rem}

\section{Example of a global Newton-Okounkov body}\label{s:ex2}

Let us consider the 3-dimensional incidence variety $Y$ from Section \ref{exthreedim} again. In this section we compute the global Newton-Okounkov 
body of $Y$ with respect to the horizontal flag as well as the global semigroup $\Gamma_{Y_\bullet}(Y)$.
The necessary computations were facilitated by the use of Sagemath.
\subsection{Integrality of the Zariski decomposition}
First of all, we note that $Y$ admits an integral Zariski decomposition. This can be deduced as follows.
It can be checked by hand that the three different triples of generators of the Mori chambers $(D_3,D_2,E_2)$, $(D_1,D_2,D_3)$ and $(D_3,E_3,E_1)$ each form a $\mathbb{Z}$-basis of $\Pic_{\mathbb{Z}}(X_w)$. Hence every integral effective divisor $D$ can be written as a $\mathbb{N}$-linear combination of the generators of its corresponding Mori chamber. But this induces  the Zariski decomposition of $D$, which proves that it is integral.

\subsection{Global Newton-Okounkov body of the surface $E_3$}

By Proposition \ref{propnormal}, we have $\Gamma_{Y_\bullet}(Y)=\text{Cone}(\Gamma_{Y_\bullet}(Y))\cap \mathbb{Z}^6$. It suffices therefore to compute $\Delta_{Y_\bullet}(Y)=\text{Cone}(\Gamma_{Y_\bullet}(Y))$ in order to determine $\Gamma_{Y_\bullet}(Y)$.
We start with computing the global Newton-Okounkov body of the surface $E_3$, with respect to the induced horizontal flag.
The divisor $E_3$ is isomorphic to the Blowup $X$ of $\mathbb{P}^2$ in one point. Since $E_2$ is an extremal ray of the effective cone which is not nef, it is the exceptional divisor. Since $D_1=E_1$ is a nef divisor which is an extremal divisor of the effective cone, it is linear equivalent to the strict transform of a line going through the blown up point. Furthermore, it follows that $D_2$ is linear equivalent to the pullback of a line of $\mathbb{P}^2$ to $X$. Hence, we get 
\begin{align*}
(E_1)^2=0 \quad (E_1\cdot D_2)=1 \quad  (E_2\cdot D_2)=0 \quad (E_1\cdot E_2)=1.
\end{align*}
Now, it follows with the help of \cite{SS16} that $\Delta_{{Y_1}_\bullet}(Y_1)$ is generated by the following vectors
\begin{align*}
(1,0,E_2), \ (0,0,D_1), \ (0,0,D_2), \ (0,1,D_1).
\end{align*}

\subsection{Global Newton-Okounkov body of $Y$}
In order to compute $C(S_1):=\text{Cone}(S_1(D_1,\dots,D_3))$, we need to intersect $\Delta_{{Y_1}_\bullet}(Y_1)$ with $\mathbb{R}^2\times\text{Cone}(\{D_1,D_2)\}$.
A computation yields to the following generators of $C(S_1)$
\begin{align*}
(0,0,D_1), \ (0,0,D_2), \ (0,1,D_1), \ (1,0,D_2),\ (1,1,D_2).
\end{align*}
Choosing the basis $E_1,E_2,E_3$ for $N^1(Y)$, the following are the defining inequalities of $C(S_1)$
\begin{align*}
-x_1+x_4\geq 0, \ x_1\geq 0,\ x_3-x_4\geq 0,\\ x_2\geq 0,\ x_1-x_2+x_3-x_4\geq 0.
\end{align*}
Now we consider the restriction morphism $q\colon \text{Cone}(S(D_1,D_2,D_3))\to C(S_1)$ which induces a linear morphism on the corresponding linear spaces.
By realizing that ${E_3}_{\vert E_3}=E_1-E_2$, the morphism $q$ can be written as
\begin{align*}
q\colon \mathbb{R}^6\to \mathbb{R}^4 \quad (a_1,\dots, a_6)\mapsto (a_2,a_3,a_4+a_6,a_5-a_6).
\end{align*}
Hence, the defining inequalities of $q^{-1}(C(S_1))$ are given by
\begin{align*}
-x_2+x_5-x_6\geq 0, \ x_2\geq 0, \ x_4-x_5+2\cdot x_6\geq 0, \\ x_3\geq 0, \ x_2-x_3+x_4-x_5+2x_6 \geq 0.
\end{align*}

The set
$\text{Cone}(S(D_1,D_2,D_3)$ is given by $q^{-1}(C(S_1))\cap \left( \{0\}\times \text{Cone}(D_1,D_2,D_3)\right)$.
The ray generators are then computed as
\begin{align*} 
&(0, 0, 0, D_3), \ (0, 0, 0, D_1), \ (0, 0, 0, D_2), \ (0, 0, 1,D_3),\\
& (0, 0, 1, D_1), \ (0, 1, 0,D_2), \ (0, 1, 1, D_2). 
\end{align*}

In order to obtain the generators of $\Delta_{Y_\bullet}(Y)$ we simply need to add the ray $(1,0,0,E_3)$ as well as $(0,1,0,E_2)$.
This yields to the following minimal set of generators of $\Delta_{Y_\bullet}(Y)$:
\begin{align*}
&(0, 0, 0,D_3), \ (0, 0, 0, D_1), \ (0, 0, 0,D_2), \ (0, 0, 1, D_3),\\  
& (0, 0, 1, D_1), \ (0, 1, 0, E_2), \ (1, 0, 0, E_3).
\end{align*}

\subsection{Generators of the global semigroup/Cox ring}

In fact, the above generators of $\Delta_{Y_\bullet}(Y)$ are actually a Hilbert basis and consequently they generate the global semigroup $\Gamma_{Y_\bullet}(Y)$.
Furthermore, we can use the above generator and Theorem \ref{thmcox} to deduce that the following sections generate the Cox ring $\operatorname{Cox}(Y)$
\begin{itemize}
\item a general section of $D_i$ for $i=1,2,3$
\item the generating section of $E_i$ for $i=2,3$
\item a section of $D_i$ for $i=1,3$ which vanishes exactly once at the chosen point of the horizontal flag and not on $E_3\cap E_2$.
\end{itemize}

\subsection{Connections to the flag variety $Fl(\mathbb{C}^3)$}
The birational morphism $p\colon X_w\to Y_{\overline{w}}$ in our example is the following:
\begin{align*}
p\colon Y\to Fl(\mathbb{C}^3)\quad  (V_1,V_2,V_2^\prime)\mapsto (V_1,V_2^\prime).
\end{align*}
The divisors $D_2=\det (\mathcal{V}_1)^*$ and $D_3=\det(\mathcal{V}_2^\prime)^*$ generate the effective cone of $Fl(\mathbb{C}^3)$.
Hence, we can compute the global Newton-Okounkov body as 
\begin{align*}
\Delta_{Y_\bullet}(Fl(\mathbb{C}^3))= \Delta_{Y_\bullet}(Y) \cap \left( \mathbb{R}^3\times \text{Cone}(D_2,D_3)\right).
\end{align*}
A computation shows that the following points are generators of the global Newton-Okounkov body $\Delta_{Y_\bullet}(Fl(\mathbb{C}^3)$:
\begin{align*}
 &(0, 0, 0, D_2),\ (0, 0, 0, D_3), \ (0, 1, 1, D_2), \ (0, 0, 1, D_3),\\
  &(1, 1, 0, D_3), \ (0, 1, 0, D_2).
\end{align*}
As in the previous case for $Y$ it turns out that the above generators are indeed a Hilbert basis for the global semigroup $\Gamma_{Y_\bullet}(Y_{\overline{w}})$.
Moreover, with the help of Theorem \ref{thmcox}, we can deduce that the following sections generate the Cox ring $\operatorname{Cox}(Fl(\mathbb{C}^3))$:
\begin{itemize}
\item A general section $s_i$ of $D_i$ for $i=2,3$.
\item A section $s_3^\prime$ of $D_3$ which vanishes exactly once at a fixed point and not on $C=E_2\cap E_3$.
\item The section of  $s_{E_2}\otimes s_{E_3}$ in $D_3$ where $s_{E_i}$ is a defining section of $E_i$ for $i=2,3$.
\item A section of the form $s_{E_2}\otimes s^\prime_1$ where $s^\prime_1$ is a section in $D_1$ which vanishes exactly once at a fixed point and not on $C=E_2\cap E_3$.
\end{itemize}

\subsection{Computation of an example for a (local) Newton-Okounkov body}
We can now use the description of the global Newton-Okounkov body to compute Newton-Okounkov bodies corresponding to special divisors.
Let us, for example, fix the divisor $D=D_1+D_2+D_3$.
Then $\Delta_{Y_\bullet}(D)=\Delta_{Y_\bullet}(Y)\cap \left( \mathbb{R}^3\times \{D\}\right)$.
A computation shows that the above Newton-Okounkov body is the convex hull of the vertices
\begin{align*}
&(1, 0, 0), \ (1, 2, 0), \ (1, 2, 2),\ (0, 1, 3),\\ &(0, 1, 0),\ (0, 0, 2),\ (0, 0, 0).
\end{align*}
This Newton-Okounkov body has been computed in \cite[Example 4.1]{HY} by different methods. 
Since $\Gamma_{Y_\bullet}(D)$ is normal, we can compute 
the Hilbert polynomial of $D$ as the Ehrhart polynomial of this polytope.
Note that the Ehrhart polynomial $P$  of a lattice polytope $\Delta\subset \mathbb{R}^d$ is the polynomial function given on integers $k\in\mathbb{N}$ by:
\begin{align*}
P(k)=\# \left(k\Delta\cap \mathbb{Z}^d\right). 
\end{align*}
It is given by
\begin{align*}
P_D(t)=5/2 t^3+11/2t^2+4t+1.
\end{align*}
\begin{figure}
\center
\includegraphics[scale=1]{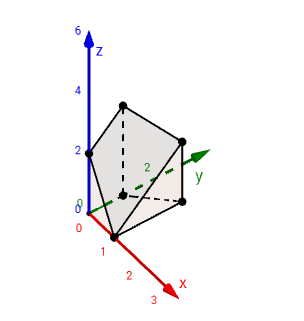}
\caption{Newton-Okounkov body of $D=D_1+D_2+D_3$}
\end{figure}

\section{Open problems and conjectures}
We end this article with some open questions and conjectures.

\subsection{Equality of moving cone and nef cone}
One of the main reasons which lead us to nice characterizations of the Mori chambers and Newton-Okounkov bodies was the fact that we have a Zariski decomposition on Bott Samelson varieties $X_w$. This 
is a consequence of the fact that $\text{Mov}(X)=\text{Nef}(X)$. Besides for surfaces there are not many varieties known which have this property. Therefore, we raise the following question.
\begin{question}
For which varieties $X$ do the cones  $\text{Mov}(X)$ and $\text{Nef}(X)$ agree?
\end{question}

\subsection{Finite generation of the global semigroup}
We have seen that the finite generation of the global semigroup $\Gamma_{Y_\bullet}(X)$ was quite restrictive, i.e. we really needed a lot of nice properties (like existence of Zariski decomposition, vanishing of cohomology) in order to establish this result.
However, it was proven in \ 
in \cite{PU16}, that for a Mori dream space $X$, there always exist a flag $Y_\bullet$ such that $\Delta_{Y_\bullet}(X)$ is rational polyhedral. It was also proven in this article that for an ample divisor $A$, the semigroup $\Gamma_{Y_\bullet}(A)$ is finitely generated.
We can now pose the following problem.
\begin{question}
Let $X$ be a Mori dream space. Does there always exist a flag $Y_\bullet$ such that the corresponding semigroup $\Gamma_{Y_\bullet}(X)$ is finitely generated?
\end{question}

\subsection{Toric degenerations}
As the finite generation of the global semigroup $\Gamma_{Y_\bullet}(X)$ is an interesting question per se, we believe that, in analogy to the fact that the finite generation of the semigroups $\Gamma_{Y_\bullet}(D)$ 
induce toric degenerations  of $X$ to $\Delta_{Y_\bullet}(D)$ (\cite{A13}), the following holds.
\begin{conjecture}
Let $X$ be a Mori dream space.
Then the finite generation of $\Gamma_{Y_\bullet}(X)$ induces a degeneration of $\text{Spec}(\text{Cox}(X))$, which is compatible with the toric degenerations of $\Gamma_{Y_\bullet}(D)$, considered in \cite{A13}.
\end{conjecture}
\subsection{Normality of the semigroup}
We have seen in Proposition \ref{propnormal} that the normality of the semigroup $\Gamma_{Y_\bullet}(X_w)$ is connected to the existence of  integral Zariski decompositions. We have also seen in the last 
section that for our example the variety $Y$ induces integral Zariski decompositions. It is also not difficult to prove that the four dimensional example $X$ from Section \ref{fourdim}, induces integral Zariski decompositions.
It is now natural to ask the following question.
\begin{question}
Under which circumstances does the Bott-Samelson variety $X_w$ induce integral Zariski decompositions?
\end{question}



\end{document}